\documentclass[10pt]{imsart}

\usepackage{amsthm,amsmath,amsfonts,amssymb}
\usepackage{ifthen}
\usepackage{stmaryrd}
\usepackage{graphicx}

\usepackage{natbib}

\usepackage{chngcntr}
\usepackage[colorlinks,citecolor=blue,urlcolor=blue]{hyperref}

\usepackage[normalem]{ulem}
\usepackage{dsfont}
\usepackage{bm}

\usepackage{enumitem}

\usepackage{xspace}

\startlocaldefs

\numberwithin{equation}{section}
\theoremstyle{plain}
\newtheorem{theorem}{Theorem}
\newtheorem{lemma}{Lemma}
\newtheorem{proposition}{Proposition}
\newtheorem{cor}{Corollary}
\theoremstyle{remark}
\theoremstyle{remark}
\newtheorem{fact}{Fact}[section]

\newtheorem{definition}{Definition}
\newtheorem{remark}{Remark}
\newcommand{\cond}{\textbf{C}\!\!}

\newcommand{\nset}{\mathbb{N}}
\newcommand{\rset}{\mathbb{R}}

\newcommand{\ind}{\mathds{1}}

\newcommand{\un}{\ind}
\newcommand{\PP}[1][]{\ifthenelse{\equal{#1}{}}{\ensuremath{\mathbb{P}}}{\ensuremath{\mathbb{P}\left( #1 \right) }}}
\newcommand{\EE}[1][]{\ifthenelse{\equal{#1}{}}{\ensuremath{\mathbb E}}{\ensuremath{{\mathbb E}\left[ #1 \right]}}}

\newcommand{\esp}{\EE} 
\newcommand{\Var}[1]{\mathbb{V}\mathrm{ar}\left[ #1 \right]}
\newcommand{\Cov}[1]{\mathbb{C}\mathrm{ov}\left[ #1 \right]}

\DeclareMathOperator*{\argmin}{arg\,min}

\DeclareMathOperator{\Span}{span}

\DeclareMathOperator{\Vector}{vec}

\newcommand\ie{\emph{i.e.}\xspace}
\newcommand\iid{\ensuremath{\mathit{i.i.d.}}\xspace}

\newcommand\cdf{\emph{c.d.f.}\xspace}

\newcommand{\ud}{\,\mathrm{d}}
\newcommand{\point}{\,\cdot\,}
\newcommand{\given}[1][{}]{\;\middle\vert\;{#1} }
\def\indep{\perp\!\!\!\!\perp}

\newcommand{\supy}{{y^+}}

\newcommand{\tailindep}[3]{\left. {#1}_\infty {\indep} {#2} \given {#3} \right.}
\newcommand{\TCI}{\textsc{TCI}\xspace}
\newcommand{\TCIG}{\textsc{TCI-G}\xspace}
\newcommand{\AMR}{\mathcal{R}_{\textrm{AM}}}
\newcommand{\AMl}{\ell_{\textrm{AM}}}

\endlocaldefs

\begin{document}

\begin{frontmatter}

  \title{Tail Inverse Regression:  dimension reduction for prediction of extremes}
\runtitle{Tail Inverse Regression for Extremes}
	\begin{aug}

\author[A]{\fnms{Anass} \snm{Aghbalou}\ead[label=e1]{anass.aghbalou@telecom-paris.fr}}
\author[C]{\fnms{Fran\c{c}ois} \snm{Portier}\ead[label=e2]{francois.portier@gmail.com}}
\author[D]{\fnms{Anne} \snm{Sabourin}\ead[label=e3]{anne.sabourin@u-paris.fr}}
\author[B]{\fnms{Chen} \snm{Zhou}\ead[label=e4]{zhou@ese.eur.nl}}
\runauthor{A. Aghbalou, F. Portier, A. Sabourin, C. Zhou}

\address[A]{LTCI, T\'el\'ecom Paris, Institut polytechnique de Paris, France.\printead[presep={,\ }]{e1}}
\address[C]{Ensai, CREST - UMR 9194, Rennes, France\printead[presep={,\ }]{e2}}
\address[D]{Université Paris Cité, CNRS, MAP5, F-75006 Paris, France\printead[presep={,\ }]{e3}}
\address[B]{Erasmus University Rotterdam, Rotterdam, Netherlands;   Tinbergen Institute, Rotterdam, Netherlands\printead[presep={,\ }]{e4}}
\end{aug}




\begin{abstract}
  We consider the problem of supervised dimension reduction with a
  particular focus on extreme values of the target $Y\in\rset$ to be
  explained by a covariate vector $X \in \rset^p$. 
    The general
    purpose is to define and estimate a projection on a lower
    dimensional subspace of the covariate space which is sufficient
    for predicting exceedances of the target above high thresholds.
    We propose an original definition of Tail Conditional Independence
    which matches this purpose.  Inspired by  Sliced Inverse
    Regression (SIR) methods, we develop a novel framework (TIREX,
    Tail Inverse Regression for EXtreme response) in order to estimate
    an extreme sufficient dimension reduction (SDR) space of
    potentially smaller dimension than that of a classical SDR space.%
   We prove the weak convergence of tail empirical processes involved
  in the estimation procedure and we illustrate the relevance of the
  proposed approach on simulated and real world data.
\end{abstract}


\end{frontmatter}


\section{Introduction}
\label{sec:intro}

Dimension reduction is a crucial matter in supervised learning
problems where the goal is to predict a \emph{dependent variable}
$Y\in\rset$ or summaries of it, when the dimension $p$ of the
\emph{covariate vector} $X\in\rset^p$ is large.  {
  In this paper we consider dimension reduction for prediction of tail
  events, by which we mean events of the kind $\{Y>y\}$, for
  arbitrarily large values of $y$.  This stylized statistical problem
  relates to a wide range of practical applications such as supervised
  anomaly detection, system monitoring with a large number of sensors,
  prediction of extreme weather conditions or financial risk
  management. For instance, in financial risk management, a typical
  concern is to identify risk factors, which will be further used to
  explain extreme events such as financial market crashes, see
  \emph{e.g.}~\cite{fama1993common,fama2015five}.  Risk factors are
  often lower dimensional functionals based on a large number of stock
  returns. Identifying such risk factors that can predict financial
  market crashes is therefore an example of dimension reduction for
  the problem of predicting tail events.  }
   
Our focus on extreme values connects our work with the field of
Extreme Value Theory (EVT) which has been successfully applied to
model tail events with potentially catastrophic impact.  Statistical
inference in this framework is performed using the most extreme
realizations of the random variable under consideration.  We refer the
interested reader to the monographs
\cite{beirlant2006statistics,dehaan+f:2007,resnick2013extreme,resnick2007heavy}.
Notice that the curse of dimensionality is particularly troublesome in
extreme value analysis where only a small fraction of the data,
reflected by the low probability $\PP[Y>y]$, is used for inference.
Before proceeding further we remark that the method proposed in this
study, although motivated by and formulated in an EVT framework, does
not rely on the minimal assumptions typically required in EVT such as
a power law decay.  It is in fact a local method related to any small
range of $Y$ and as such, it could be easily adapted to tackle the
problem of dimension reduction for prediction of $Y$ within low
probability regions of other shapes.  However in view of the
importance of applications towards risk management, we concentrate on
this specific tail region.\\~

\noindent {\bf Dimension reduction in EVT.} The subject of dimension
reduction for extremes has inspired numerous recent works. The vast
majority of them are devoted to the unsupervised setting, \ie
analyzing the extremes of a high dimensional random vector. Such
studies can be divided into the following categories: clustering
methods (\cite{chautru2015dimension,chiapino2019vizu,janssen2019k}),
support identification,
(\cite{goix2016sparse,goix2017sparse,chiapino2016feature,chiapino2019identifying,simpson2020determining,meyer2019sparse}),
Principal Component Analysis of the angular component of extremes
(\cite{cooley2016decompositions,jiang2020principal,drees2021principal}),
and graphical models for extremes
(\cite{hitz2016one,engelke2020graphical,asenova2021inference}); see
also \cite{engelke2020sparse} and the references therein.

By contrast, our approach takes place in the supervised setting.  Our
main informal assumption is that \emph{ a low dimensional orthogonal
  projection $PX$ is sufficient for {
    predicting extreme values of $Y$}}. In other words the extreme
values of $Y$ can be entirely explained by a limited number of linear
combinations of the components of $X$. {
  In this setting, the only existing works are, to our best knowledge,
  \cite{gardes2018tail} and \cite{bousebata2021extreme}.  In
  \cite{gardes2018tail}, the informal assumption emphasized above is
  made precise by a specific notion of \emph{tail conditional
    independence}, reported in Equation~\eqref{eq:defGardes}
  below. Dimension reduction is considered under this
  condition. \cite{gardes2018tail} demonstrates the usefulness of such
  a reduction for statistical estimation of large conditional
  quantiles. Even though we follow in the footsteps
  of~\cite{gardes2018tail} in terms of informal goal, our framework
  differs significantly from \cite{gardes2018tail}'s on several key
  aspects. First, the specific definition of tail conditional
  independence that we propose (See
  Definition~\ref{def:tail_cond_indep} in
  Section~\ref{sec:tail-CI-central space}) is not equivalent to
  \cite{gardes2018tail}'s
  condition~\eqref{eq:defGardes}.
  We carry out  an in-depth
  comparison of both conditions  and we show that neither one of them
  implies the other, in Section~\ref{sec:compareGardes_long} from the supplementary material.  Second, our assumption is motivated by a
  downstream task (predicting the occurrence of a tail event) which is
  different from, although related to the one motivating
  \cite{gardes2018tail} (estimation of extreme conditional quantiles).
  Third, the statistical guarantees brought by \cite{gardes2018tail}
  are obtained under the assumption that the dimension reduction space
  is already known.  In the cited reference an estimation method is
  indeed proposed for the dimension reduction space, however its
  statistical properties are only analyzed \emph{via} simulations.
  Instead, we bring statistical guarantees regarding the estimation of
  a sufficient projection subspace itself.  We discuss qualitatively
  the positive impact it may have for prediction of tail events in
  Remark~\ref{rem:TCI_AMrisk}.  Lastly, the computational cost of
  TIREX depends only polynomially on the ambient dimension $p$, which
  is not the case with the current estimation method in
  \cite{gardes2018tail}, as discussed in
  Section~\ref{sec:experiments}.  }
 
Another  study related to our work is the recently published
paper \cite{bousebata2021extreme}, where the authors adopt a
partial least square strategy to uncover the relation between linear
combinations of covariates and the extreme values of the target. Their
model assumptions differ from ours substantially: the inverse
regression model assumed in \cite{bousebata2021extreme}
implies a single-index relationship
between extreme values of the response and the covariates. In
addition, the model requires regular variation of the dependent variable $Y$
and of the link function.  Lastly, the model relies on finite variance of $Y$. 
In contrast, our approach is somewhat `free' from most
restrictions on the distribution of $(X,Y)$ except from the well-known linearity
condition and constant variance condition,
typically needed for SIR. Such conditions concern only the distribution of the
covariates.  Since we do not impose regular variation, 
we can handle not only
thin-tailed but also extremely heavy-tailed dependent variables with no finite variance or even mean.\\~

\noindent{\bf Sufficient Dimension Reduction and inverse methods.}
The underlying assumption of a sufficient linear projection subspace
has been formalized under the notion of Sufficient Dimension Reduction
(SDR) space (\cite{cook2009regression}). Many classical approaches to
supervised dimension reduction rely on a linear regression model
between $X$ and $Y$. This is the case \emph{e.g.} for Principal
component regression (\cite{hotelling1957relations}), Partial least
squares (\cite{wold1966estimation}), Canonical correlation analysis
(\cite{thompson1984canonical}) or penalized methods with sparsity
inducing regularization such as the Lasso
(\cite{jenatton2011structured}).  Differently, \emph{SDR} builds upon
a \textit{linear dimension reduction} assumption: only a small number
of \textit{linear} combinations of covariates is useful for predicting
the dependent variable. In other words, there exists a linear subspace
$E$ (a SDR) of a moderate dimension $d\leq p$, such that
\begin{align}\label{def:sdr_space}
  \PP[Y\le t \given X ] = \mathbb P(Y\leq t |PX), \qquad \forall t\in\rset, \quad \text{almost surely}, 
\end{align}
where $P$ is the orthogonal projector on $E$, \ie $Y$ depends on $X$
only through $PX\in \mathbb R^d$. This framework relies heavily on the
notion of conditional independence
\cite{dawid:1979,constantinou+d:2017}: Condition \eqref{def:sdr_space}
characterizes the fact that $Y$ is conditionally independent from $X$
given $PX$. One major advantage of this approach is that it strikes a
balance between interpretability of the dimension reduction based on
linear operations and flexibility of the generative model -- no
assumption is made regarding the dependence structure between $PX$ and
$Y$.

Under the assumption that there exists a non trivial subspace $E$ such
that~\eqref{def:sdr_space} holds, a natural idea is to estimate such a
subspace first, and then use only the variable $PX$ to predict $Y$,
thus reducing the dimensionality of the regression problem.  The
estimation problem based on SDR can also be viewed as a specific case
of semi-parametric M-estimation
(\cite{delecroix2006semiparametric}). Alternatively, one may consider
derivative based methods, relying on the fact that the gradient of the
regression curve belongs to $E$
(\cite{hardle1989investigating,hristache2001structure,xia2007constructive,dalalyan2008new}). Recently,
the framework of Reproducing Kernel Hilbert Spaces (RKHS) has been
employed to estimate SDR spaces by means of covariance operators
(\cite{fukumizu2004dimensionality,fukumizu2009kernel}).

The family of methods to which our work relates most is the inverse
regression paradigm initiated by \cite{li:1991}, including the Sliced
Inverse Regression (SIR) strategy and its second order variant Sliced
Average Variance Estimate (SAVE) (\cite{cook+w:1991}).  The main idea
underlying these methods is that under appropriate assumptions the
inverse regression curve $\esp[X| Y]$ and its second moment variant --
the columns of the conditional covariance matrix $\Var{X|Y}$ -- almost
surely belong to the minimal SDR. Cumulative slicing estimation
(CUME), proposed in \cite{zhu+z+f:2010} and further analyzed in
\cite{portier:2016}, 
aims at recovering the largest possible subspace of the minimal
SDR. It is achieved by estimating the conditional expectation and
variance of $X$, conditioning on `slices' of the target $Y$, in the
form of $\un\{Y<y\}$, and then aggregating such conditional
expectations and variances by integration with respect to $y$.

A well-known restriction of the SIR strategy is that it relies on a
so-called \emph{linearity condition} (LC) regarding the covariates,
namely equation~\eqref{eq:LC} in the next section, see
\cite{hall:1993} for a justification.  The required condition is
satisfied in particular if the covariates form an elliptical random
vector or are independent
(\cite{cook2009regression,eaton1986characterization}). There are
various extensions of SIR permitting to overcome this
restriction. Using RKHS, it has been proposed to transform the data in
a way that LC is approximately satisfied
(\cite{wu2008kernel,yeh2008nonlinear}).  Another possibility allowing
to depart from elliptical covariates is to apply the SIR methodology
and its higher order variants to score functions of the explanatory
variables (\cite{babichev2018slice}).  Finally, the high dimensional
case $p>n$ calls for regularization methods which permit in addition
to perform feature selection (\cite{li+y:2008}). All these extensions
are out of the scope of the present paper, in which we restrict
ourselves to the original SIR and SAVE methods, thus leaving room for
several improvement in further works. For estimation purposes we
consider a variant of CUME.\\~

\noindent{\bf Contributions and outline.} {
  Our contributions are twofold.  First, we develop in
  Section~\ref{sec:tail-CI-central space} a modified version of
  \cite{gardes2018tail}'s probabilistic setting regarding tail
  conditional independence.  In particular we explain in
  Remark~\ref{rem:TCI_AMrisk} the relevance of our definition for the
  purpose of predicting tail events and its connections to the
  statistical learning framework of imbalanced classification.  We
  discuss thoroughly the distinctions between the two alternative
  definitions for tail conditional independence in
  Section~\ref{sec:compareGardes_long} where we also provide examples
  of models satisfying one or the other.  Second, we show in
  Section~\ref{sec:tail-SIR} that our definition permits to extend
  inverse regression principles and methods to this extreme values
  setting (theorems~\ref{th:SIR_extreme},~\ref{th:SAVEx}).  We derive
  an asymptotic analysis for our proposed estimation strategy TIREX
  stemming from inverse regression, using specific tools from the
  theory of empirical processes
  (Section~\ref{sec:estimationWithUniform}).  We illustrate the finite
  sample performance of TIREX with simulated and real world data sets
  in Section~\ref{sec:experiments}, in particular we demonstrate
  empirically the usefulness of TIREX for tail events prediction.  The
  code developed for TIREX is available
  online \footnote{\url{https://github.com/anassag/TIREX}}
 and some technical proofs and additional comments are deferred to the
 supplementary material.

 We start-off in
Section~\ref{sec:background-SIR} by recalling the necessary background
regarding conditional independence of random variables, SDR spaces,
and inverse regression.  }
 

\section{Background: dimension reduction space and Sliced Inverse Regression}\label{sec:background-SIR}

Conditional independence of  random variables $Y$ and $V$ given $W$ is defined \emph{e.g.} in  \cite{constantinou+d:2017} as follows: the conditional distribution of $Y$ given $(V,W)$ is the same as the conditional distribution of $Y$ given $W$, almost surely. Several  characterizations are recalled below, the equivalence of which  are proved in \cite{constantinou+d:2017}, Proposition~2.3. 
\begin{definition}[conditional independence]\label{def:conditional_indep}
  Let $Y,V,W$ be random variables defined on a probability space
  $(\Omega, \mathcal F , \mathbb P)$ {
    and taking values in arbitrary measure spaces}.  The variables $Y$
  and $V$ are called conditionally independent given $W$, a property
  denoted by $Y\indep V \;|\; W$, if the equivalent conditions below
  are satisfied.
  \begin{enumerate}
  \item [(CI-1)]  For all $A_Y \in \sigma(Y)$,  $\PP[Y \in A \given V,W] = \PP[Y \in A\given W]$, almost surely.
  \item [(CI-2)] 
    For all real-valued functions $f $ and $ g$, measurable and bounded, 
   \begin{align*}
\EE[f(Y) g(V) \given W] = \EE[f(Y)  \given W] \EE[ g(V) \given W]  ,\quad \text{a.s.} 
    \end{align*}
  \item [(CI-3)]For any real-valued function $g$, measurable and bounded, 
 \begin{align*}
              \EE[ g(V) \given Y,W ] = \EE[g(V)  \given W] , \quad \text{a.s.}
         \end{align*}
  \end{enumerate}
\end{definition}

  Notice that the existence of regular versions of conditional
  probability distributions is not required in
  Definition~\ref{def:conditional_indep}. However in this paper, $Y$
  is real valued, thus the existence of such a regular version for the
  conditional distribution of $Y$ given $(V,W)$ is granted. As a
  consequence we may write, without additional precautions,
  expressions of the kind `$\PP[Y\in A \given V=v,W=w]$'. The latter
  quantity is defined as the value of the conditional probability
  kernel at point $( (v,w), A
  )$.

 In the context of supervised dimension reduction, we consider $V=X$
 and search for a projection $W = PX$ of $X$ on a lower dimensional
 subspace $E$ satisfying the above conditions.  We assume for
 simplicity that the covariance matrix $\Sigma = \Cov{X}$ is
 invertible and for ease of presentation we introduce a standardized
 covariate vector $Z = \Sigma^{-1/2}(X - m)$ where $m = \EE[X]$.  We
 consider in the remaining of this paper the problem of regressing $Y$
 on $Z$, which amounts to assuming that both $m$ and $\Sigma$ are
 known, so that the vector $Z$ is observed.  Thus $\Var{Z} = I_p$ and
 $\esp[Z] = 0$.  A SDR space (\cite{cook2009regression,cook+n:2005})
 is a subspace $E$ of $\rset^p$ such that
 $\left. Y\indep Z\given P Z \right.$ where $P$ is the orthogonal
 projection on $E$, which is equivalent to
 condition~\eqref{def:sdr_space} in the introduction.  Our results
 easily extend to general covariates $X$ (see \emph{e.g.}
 \cite{cook+w:1991}) at the price of an additional notational burden,
 see Section~\ref{sec:nonstandard} in the supplementary material.
 Notice already that in terms of non-standardized covariates $X$, a
 subspace $\tilde E$ of $\rset^p$ with associated orthogonal projector
 $\tilde P$ is a SDR space for the pair $(X,Y)$ if and only if
 $\tilde E = \Sigma^{-1/2} E$ where $E$ is a SDR space for $Z$.

 A central space is a  SDR  subspace $E_c$ for the pair $(Z,Y)$ of minimal dimension. In our context of finite dimensional covariates  a central space always exists since the ambient space $\rset^p$ itself is a SDR space. Uniqueness is not guaranteed in general but holds true under mild assumptions ensuring that an intersection of SDR spaces is a SDR space (see \emph{e.g.} \cite{portier+d:2013}, Theorem~1). In such a case one may refer without ambiguity to \emph{the} central space.

 First and second order inverse methods, respectively named SIR
 (\cite{li:1991}) and SAVE (\cite{cook+w:1991}) are two of many
 methods to estimate SDR spaces. Both rely on the fact that under
 appropriate assumptions detailed below, first and second moments of
 the covariate vector, conditioning upon the target, belong to a SDR
 space. In the sequel, $E$ is a SDR space, and $P$ 
 denotes 
 the orthogonal projection on $E$.  Then $Q = I_p-P$ is the orthogonal
 projection on $E^\perp$, the orthogonal complement of $E$.  The
 required conditions are the Linearity Condition (LC):
\begin{equation}
  \label{eq:LC}
  \EE[Z  \given  PZ] = PZ \quad \text{a.s.}
\end{equation}
and the additional Constant Conditional Variance  (CCV), 
\begin{equation}
  \label{eq:CCV}
  \Var{Z|PZ } \text{ is constant} \quad \text{a.s.}
\end{equation}
Under both LC and CCV, we have that
$ \EE [\Var{Z|PZ }] = \EE [ZZ^T] - \EE [PZ(PZ)^T] = I_p- P$ and
therefore the constant matrix in \eqref{eq:CCV} is necessarily the
projection $Q = I_p-P$ on the orthogonal complement of $E$.

Notice that LC and CCV depend on an unknown SDR space. Assuming that
LC holds for all orthogonal projectors is in fact equivalent to
assuming that the covariate vector $Z$ is spherically symmetric, \ie
$Z = \rho U$ where $\rho\indep U$, $\rho$ is a non negative random
variable and $U$ is uniformly distributed over the unit sphere of
$\mathbb R^p$, as proved in \cite{eaton1986characterization}. Among
spherical variables with finite second moment, CCV is equivalent to
being Gaussian (\citep[Theorem 4.1.4]{bryc2012normal}).

The following proposition in \cite{li:1991} encapsulates the main idea of SIR. We give below the (classical) proof for the sake of completeness.   
\begin{proposition}[SIR principle]\label{prop:sir_principle} 
  If $E$ is an SDR space for which LC~\eqref{eq:LC} is satisfied, then $Q ( \EE[Z  | Y]) = 0$ a.s., 
    that is, $\EE[Z | Y]  \in  E$ a.s.
  \end{proposition}
  \begin{proof}
    By the tower rule from conditional expectation,
    \begin{align*}
      \EE[Z\given Y] & =
                       \EE[\EE\big(Z \,|\, Y, PZ\big) \given Y ]
                       = \EE[\EE( Z \,|\, PA) \given Y)] \\
                     & = \EE[PZ 
                       \given Y] = P\,\EE[Z\given Y]\end{align*}
                   where the second equality comes from conditional independence and the third one follows from the linearity condition~\eqref{eq:LC}. Thus $Q\EE[Z|Y] =0$. 
\end{proof}
  The SIR method advocated first by \cite{li:1991} consists in
estimating first conditional expectations
$C_h = \EE[Z \given Y \in I(h)]$, $h= 1,\ldots, H$, where
$I(h), h=1,\ldots, H$ are called slices and form a partition of the
sample range of $Y$ (or the support of the density function if $Y$ is
continuous). From Proposition \ref{prop:sir_principle}, those
estimates lie in the vicinity of the SDR space. Next, performing a
Principal Component Analysis (PCA) on the $C_h$'s, one obtains a good
approximation of $E$. 
More precisely, the SIR estimate of $E$ is given by the largest
eigenvectors associated to the SIR matrix, 
\begin{align*}
M _{\text{SIR}} = \sum_{h=1}^H p_h^{-1} C_hC_h^T,
\end{align*}
where $p_h = \PP[Y\in I(h)]$; see \cite{li:1991}.
Various estimation procedures of SDR spaces are proposed in \cite{cook+n:2005,zhu+z+f:2010}. In the latter reference, the matrix
\begin{align}\label{eq:M-cume}
M _{\text{CUME}}  = \EE\big[m(Y) m(Y) ^T\big] ,
\end{align}
with $m(y)  = \EE[Z \un\{  Y \le y\} ]$, is introduced as an alternative to the SIR matrix.  One advantage of this approach is that the slicing parameter $h$ is no longer needed.  In addition the estimate of the matrix $M _{\text{CUME}}$ benefits from the  aggregating effect of the expectation sign which is typically associated with variance reduction.

A pitfall of SIR is that it is not guaranteed that the $C_h$'s span the entire space $E$, so that SIR may be inconsistent. 
This may happen in particular  when the regression function $\EE[Y|Z]$ admits some symmetry properties \cite[Remark 4.5]{li:1991}, a phenomenon referred to as the SIR pathology.
In this case, \cite{li:1991} and \cite{cook+w:1991} recommend to use higher order moments such as the conditional variance of $Z$ given $Y$ to obtain a second order matrix with wider range.
This second order method requires that CCV~\eqref{eq:CCV} is satisfied in addition to LC, in which case the following result holds. Here and throughout,  $\Span(M)$ stands for the column space of matrix $M$. 
\begin{proposition}[SAVE principle]\label{prop:SAVE}
  If E is an SDR space for which  LC~\eqref{eq:LC} and CCV~\eqref{eq:CCV} are satisfied, then
  \[
Q\Big( \EE\big[ZZ^\top  \,|\, Y\big] -I_p \Big) = 0 \qquad a.s., 
\]
in other words 
$\Span\big(\EE\,\![ZZ^\top  \,|\, Y] -I_p)\subset E \quad  a.s.$
\end{proposition}
\begin{proof}
   We reformulate here the arguments of \cite{cook+w:1991} in our notational framework for convenience.  An immediate consequence of assumptions~\eqref{eq:LC} and \eqref{eq:CCV} is that $\EE[ZZ^T|PZ] = Q + PZZ^TP$. From a conditioning argument and the conditional independence assumption, $\EE[ZZ^T|Y] =  Q + P\EE[ZZ^T|Y]  P$. Rearranging gives $\EE[ZZ^T|Y] - I_p =   P(\EE[ZZ^T|Y] - I_p )P$, thus $Q (\EE[ZZ^T|Y] - I_p ) = 0$. 

\end{proof}
 Notice that propositions~\ref{prop:sir_principle} and~\ref{prop:SAVE} together imply that
 $ Q(\Var{Z  \given Y} -I_p) = 0$. 
Finally for estimation purpose the extension of the CUME method to the second order framework is termed CUVE (cumulative variance estimation) by \cite{zhu+z+f:2010}.  In the case of standardized covariates,  it consists in estimating the matrix
 $ M_{\text{CUVE}} = \EE\big[ W(Y) W(Y)^\top \big] $, 
where $ W(y) = \Var{ Z \un\{ Y\le y\} } - F_Y(y)I_p $ is a second order moments matrix  which column space is included in $\tilde E$. The latter fact is obtained  by a slight modification of the argument leading to the SAVE principle. 



\section{Tail conditional independence, Extreme SDR space }\label{sec:tail-CI-central space}

\subsection{Definition for Tail Conditional Independence}\label{sec:defineTCI}
  The focus on the largest values of the target variable $Y$ suggests
  to weaken the classical definition of conditional independence, so
  that the equivalent conditions (CI-1)-(CI-3) hold only for $Y$
  exceeding a high threshold tending to its right endpoint.  Namely,
  in a similar (but different) manner as in~\cite{gardes2018tail} we
  define tail conditional independence as a variant of condition
  (CI-1) from Definition~\ref{def:conditional_indep}.  In the sequel
  the right endpoint (\ie the supremum) of the support of the random
  variable $Y$ is denoted by $\supy$. The limits as $y\to \supy$ as
  understood as the limits as $y\to \supy, y<\supy$.  We assume that
  $\PP[Y>y]\xrightarrow[]{}0$ as $y\to \supy$, in particular we
  exclude the case of point masses at $\supy$.

\begin{definition}[Tail Conditional Independence (\TCI)]\label{def:tail_cond_indep}
  Let $Y,V,W$ be random variables defined on $(\Omega, \mathcal F , \mathbb P)$. We assume that $Y$ is real valued, Borel measurable, while $V$ and $W$ take their values in arbitrary measure spaces. We say that $Y$ is \textit{tail conditionally independent from $V$ given $W$} and write $Y{}_{\infty}{ \indep} V | W$, if 
  \begin{align} \label{eq:newdefTCI}
\frac{\EE \, \big|\, \PP[ Y>y \given V, W ] - \PP[Y>y \given W ]\,\big| }{\PP[Y>y]}  &\xrightarrow[y\to \supy]{} 0. 
  \end{align}
\end{definition}
Contrary to conditional independence, tail conditional independence
is not symmetric: $\tailindep{Y}{V}{W}$ does not imply that
$\tailindep{V}{Y}{W}$.

  In \cite{gardes2018tail}'s work, tail conditional
  independence is defined in a somewhat more technical manner, see
  Definition~1 from the cited reference.  However a necessary
  condition (see Equation (2) in that paper) is the almost sure
  convergence of the $\sigma(V,W)$-measurable ratio,
  \begin{equation}
    \label{eq:defGardes}
     \frac{\PP[Y> y \given V,W] - \PP[Y>y \given W]}{\PP[Y>y \given W]} \xrightarrow[y\to \supy]{} 0, \quad \text{a.s.} 
  \end{equation}
  In the sequel we refer to our notion of tail conditional
  independence defined in \eqref{eq:newdefTCI} as \TCI, while we write
  \TCIG to refer to L. Gardes' condition~\eqref{eq:defGardes}.
Both definitions are
motivated by similar but different downstream tasks, namely prediction
of extreme values for \TCI in connection to the AM risk criterion (see
Remark~\ref{rem:TCI_AMrisk} below), versus estimation of large conditional
quantiles (see Section 3.1 in~\cite{gardes2018tail}). 


In Subsection~\ref{sec:MixtureGeneric} below we work out a generic example where \TCI holds and on this occasion, we discuss briefly the differences between \TCI and \TCIG. In order not to interrupt the flow of ideas a more thorough comparison between the two definitions is relegated to the supplementary material (Section~\ref{sec:compareGardes_long}).

In practice 
\TCI allows for an extension of the SIR framework to
handle extreme values (Section~\ref{sec:tail-SIR}). Whether it is
possible to obtain a similar extension with 
\TCIG is an open
question. We conjecture a negative answer because our Tail Inverse
Regression principles theorems~\ref{th:SIR_extreme},~\ref{th:SAVEx}
rely on a specific consequence of \TCI, namely
Property (iii)  
from
Proposition~\ref{prop:equi_tail_cond_exp} below.
In spirit our definition for \TCI and the subsequent Tail inverse
regression framework developed in Section~\ref{sec:tail-SIR} below is
compatible with the main notions underlying graphical models for
extremes (\cite{engelke2020graphical}) and One component regular
variation (\cite{hitz2016one}). These connections are further detailed
in Remarks~\ref{rem:tailDistrib_SIrex}
and~\ref{rem:relationships-graphical} from Section~\ref{sec:tail-SIR}.

Meanwhile the next remark sheds light on the
  relevance of the proposed definition of \TCI for statistical
  learning applications.
  
  \begin{remark}[\TCI and Imbalanced Classification]\label{rem:TCI_AMrisk}
    Predicting exceedances over arbitrarily high thresholds $y$ may be
    viewed as a family of binary classification problems indexed by
    $y$.  Indeed for fixed $y$, consider the binary target
    $T = \un\{ Y>y\}$ with marginal class probability
    $\pi = \pi_y = \PP[Y>y]$. The goal is thus to predict $T$, by
    means of the covariate vector $X = (V,W) $ where
    $V\in \rset^{p-d}, W \in \rset ^d$.  As $y\to\supy$, $\pi_y\to
    0$. This is a typical instance of \emph{class imbalance}, a well
    documented potential issue in binary classification which has been
    the subject of several works in the statistical learning
    literature, see \emph{e.g.} the recent papers
    \cite{pmlr-v28-menon13a} or \cite{pmlr-v119-xu20b} and the
    references therein.  A classifier is a binary function $h$ defined
    on $\rset^p$.  Given a family of candidate classifiers
    $h \in \mathcal{H}$ the goal is to select a `good' candidate based
    on a training set and an appropriate notion of a theoretical risk
    and its empirical counterpart.  When $\pi$ is so close to zero
    that the probability of a classification error
    $\PP[h(X)\neq T, T=1]$ is negligible compared with
    $\PP[h(X)\neq T, T=0]$, the traditional $0-1$ risk
    $R(h ) = \PP[h(X) \neq T]$ is driven by the latter term and tends
    to favor the trivial classifier $h \equiv 0$.  One standard
    approach aiming at granting more importance to the minority class
    when required by the application context (\emph{e.g.} if the event
    $\{T=1\}$, although rare, has an overwhelming impact) is to
    consider 
the \emph{Arithmetic Mean Risk} (AM risk in short), see \emph{e.g.}~\cite{pmlr-v28-menon13a}, 
\begin{equation}
  \label{eq:AMrisk}
\AMR(h) = \frac{1}{2} \Big[\PP[h(X) = 1 \given T=0 ] +  \PP[h(X) = 0\given T=1 ] \Big].    
\end{equation}
Generalizations to arbitrary weight vectors are considered in
\cite{pmlr-v119-xu20b}.  In a dimension reduction context consider the
classes
  \begin{align*}
    \mathcal{H}& = \{h: \rset^p \to \{0,1\}, \text{ measurable w.r.t. } \mathcal{B}(\rset^p) \}\;,\\
    \mathcal{H}_W & = \{h \in \mathcal{H}: \;  \forall (v,w)\in\rset^{p-d}\times \rset^d, \; h(v,w) = \tilde h(w), \; \tilde h  \text{ is  measurable w.r.t. } \mathcal{B}(\rset^d) \}.
  \end{align*}
  Let us refer to the classification problem attached respectively to
  $\mathcal{H}$ and $\mathcal{H}_W$ as the \emph{full problem } and
  the \emph{reduced problem}.  The Bayes classifier for each problem
  are respectively minimizers of the AM risk over the full family
  $\mathcal{H}$ and the reduced one
  $\mathcal{H}_W$, 
\[
 h^* \in \argmin_{h \in \mathcal{H}} \AMR(h) \;;\qquad h^*_W  \in \argmin_{h \in \mathcal{H}_W} \AMR(h) . 
\]
The main ingredient of the subsequent analysis are the regression
functions $\eta(x) = \PP[T=1\given X=x]$ and
$ \eta_W(w) = \PP[T=1 \given W = w].  $ A modification of standard
arguments (see the supplementary material,
Section~\ref{sec:proof_rem1}) 
yields explicit expressions for the Bayes classifiers
$h^*(x) = \un\{\eta(x)> \pi \}$,
$ h^*_W(x) = \un\{ \eta_W(w) >\pi \}$. In addition the Bayes risks are
\begin{equation}
  \label{eq:BayesAMRisks}
  \begin{aligned}
      \AMR (h^*) &
      =  \EE[ \min\Big( \frac{\eta(X)}{\pi}, \frac{1 - \eta(X)}{1-\pi} \Big) ] \;; \\
      \AMR (h^*_W)&
      =  \EE[ \min\Big( \frac{\eta_W(W)}{\pi}, \frac{1 - \eta_W(W)}{1-\pi} \Big) ].
  \end{aligned}
\end{equation}
     
Because $\mathcal{H}_W\subset \mathcal{H}$ we must have
$\AMR (h^*_W) \ge \AMR(h^*)$. The difference between the two may be
seen as a bias term: the price to pay for dimension reduction. Indeed
for any random choices $\widehat h\in \mathcal{H}$,
$\widehat h_W \in \mathcal{H}_W$, which are typically the outputs of a
statistical learning algorithm applied respectively to the full
covariate space and the reduced one, the excess risk for the reduced
problem decomposes as
  \begin{align*}
    \AMR(\widehat h_W) - \AMR(h^*)
    & =  \underbrace{\AMR(\widehat h_W) - \AMR(h^*_W)}_{A}  +
      \underbrace{\AMR(h^*_W) - \AMR(h^*)}_{B}\,.
  \end{align*}
  The first term ($A$) in the right-hand side is the excess risk
  stemming from the particular choice of the learning algorithm, which
  typically increases with the dimension of the input $W$.  In
  particular when $p-d$ is large, the excess risk term $A$ will be typically less than 
    its counterpart in the full problem
  $\AMR(\widehat h) - \AMR(h^*)$.   The
  second term~($B$) is the bias term above mentioned. The
  bias-variance compromise is in favour of dimensionality reduction
  \emph{via} projection on the second variable $W$ whenever
  $A+B \le \AMR(\widehat h) - \AMR(h^*)
  $. 

  We now derive an upper bound on the bias term $B$ which is closely connected to our definition of \TCI. Notice that for any finite set $\mathcal{X}$ and any pair of real functions $(f,g)$ it holds that
  $
|\min_{x\in \mathcal{X}} f(x) - \min_{x\in \mathcal{X}}g(x) | \le \max_{x\in \mathcal{X}} |f(x) - g(x)|. 
 $
  This, combined with~\eqref{eq:BayesAMRisks} above and  Jensen inequality, implies that 
    \begin{align}
      B = \AMR (h^*_W) - \AMR (h^*)
      &\le \EE \, \Big| \min\Big( \frac{\eta_W(W)}{\pi}, \frac{1 - \eta_W(W)}{1-\pi} \Big)  - \min\Big( \frac{\eta(X)}{\pi}, \frac{1 - \eta(X)}{1-\pi} \Big) \Big|   \nonumber       \\
      &\le \EE\Big\{ \max \Big( \frac{\eta(X) - \eta_W(W)}{\pi}, \frac{(1-\eta(X)) -(1  - \eta_W(W))}{1-\pi}  \Big)\Big\}\nonumber\\
     & = \EE\,\Big| \frac{\eta(X) - \eta_W(W)}{\pi}\Big|, \label{eq:boundBiasDimensionReduction}
    \end{align}
    where the latter identity holds whenever $\pi\le 1/2$. 
    Now, with $T = \un\{Y>y\}$, 
    \begin{align*}
      \EE\,\Big| \frac{\eta(X) - \eta_W(W)}{\pi}\Big|
       =  \frac{ \EE\,\Big|\PP[Y>y\given V,W] - \PP[Y>y \given W] \Big|}{\PP[Y>y]} . 
    \end{align*}
    One recognizes the \TCI criterion in the latter expression.  Thus
    \TCI means that the bias term $B$ vanishes as $y\to\supy$, so that
    \emph{projection on $W$ is relevant for the problem of predicting
      the rare event $\{Y>y\}$}, for large values of $y$.  The
      cut-off value $y$ above which $\AMR(\hat h_W)\le \AMR(\hat h )$,
      that is $A+B \le \AMR(\hat h ) - \AMR(h^* )$ (in expectation or
      with high probability), depends on two main factors: (i) the
      rate of convergence of $B_y$ to zero and (ii) the sensitivity of
      the learning algorithm to the curse of dimensionality for a
      given sample size. Indeed both excess risks
      $\AMR(\hat h ) - \AMR(h^* )$ and
      $\AMR(\hat h_W ) - \AMR(h^*_W )$ typically converge to zero (in
      expectation or in probability) with the sample size, at a
      different rate which depends on the respective dimensions
      $p,d$. 
    Precise quantification of
    this cut-off point  for specific learning algorithms and finite sample
    sizes is outside the scope of the present paper and left for
    future research.
 \end{remark}

\subsection{Examples and discussion} 
~\label{sec:MixtureGeneric}
In this section we provide a generic example based on a mixture model where the
  \TCI condition~\eqref{eq:newdefTCI} is satisfied under mild assumptions. We discuss an alternative additive model in Remark~\ref{rem:additive}. We consider several particular instances of the generic mixture  model and on this occasion  we discuss the similarities and differences between \TCI and the  \TCIG condition~\eqref{eq:defGardes} proposed
  in~\cite{gardes2018tail}. Some technical proofs are deferred to Section~\ref{sec:compareGardes_long} in the supplementary material, as well as additional comments, examples and counter-examples allowing for a better understanding of the differences between the two definitions.  


Our leading example  is constructed as follows: Let the target $Y$ be distributed  according to a mixture 
 \[   Y = B Y_1 + (1-B) Y_2\,, \]
 where $B$ is a Bernoulli variable with parameter $\theta\in(0,1)$,
 and $Y_1,Y_2$ are real variables defined through their conditional
 survival functions
  \begin{align*}
    S_1(y,V) & = \PP[Y_1>y\given V] \;,\qquad
    S_2(y,W)  = \PP[Y_2> y \given W ].  
  \end{align*}
  Here, the covariate variables $V,W$ are respectively valued in
  $\rset^{p-d}$ and $\rset^d$ with marginal distributions that we
  denote by $P_V$ and $P_W$.  The full covariate vector is
  $X= (V,W) \in \rset^p$.  We assume that the variables $(B,V,W)$ are
  independent.  Notice that independence between $V$ and $W$ ensures
  that the Linearity Condition and Constant Conditional Variance
  condition are automatically satisfied.
  In this context, straightforward calculations (as detailed in the supplementary material, Section B) show that
  $$
\frac{\left|\PP[Y>y\given V,W] - \PP[Y>y\given W] \right| }{\PP[Y>y]} =  \frac{\theta(S_1(y,V) -S_1(y) )}{
    \theta S_1(y) + (1-\theta)S_2(y)}.  
  $$
  The \TCI condition is that the expectancy of the above ratio vanishes as $y\to\supy$ and it is not difficult to imagine several models for $(Y_1,V)$ and $(Y_2,W)$ for which it is the case, as exemplified below. 

  \begin{remark}[Variant: additive model]\label{rem:additive}
    The  mixture model  described here is by no means the only option to construct examples of variables$ (Y,V,W) $ satisfying the \TCI assumption. Another natural example is  an additive model $Y = Y_1 +Y_2$, where $Y_1$ and $Y_2$ are respectively driven by $V$ and $W$, while $Y_1$ has lighter tails than $Y_2$. The mathematical derivations are somewhat more intricate because convolutions are involved instead of sums of distribution functions. However special cases can be worked out. In the supplementary material we consider $Y_1 = V \in \rset$, $Y_2 = W\zeta \in \rset$ where $\zeta$ is heavy-tailed and $V, W$ have a compact support which is bounded away from $0$ and we show that   \TCI\ holds. More general statements might be obtained using results regarding sums of regularly varying random variables (\cite{jessen2006regularly}). We leave this question to further works.
  \end{remark}
  As an example in the generic mixture model described above,  consider  the case where $Y_1$ and $Y_2$ are themselves
  defined as multiplicative mixtures
  \begin{align}
    Y_1  = \sum_{i=1}^{p-d} M_i^{(1)} V_i \epsilon_i \;, \qquad
    Y_2  = \sum_{j=1}^{d} M_j^{(2)} W_j \zeta_j \label{eq:defY12_mixture} \,,
  \end{align}
  where $M^{1} = (M_1^{1}, \ldots, M_{p-d}^1)$ is a multinomial vector
  with weight parameter
  $\pi^{1} = (\pi^{1}_1, \ldots, \pi^{1}_{p-d})$, that is
  $\sum_{i=1}^{p-d} M_{i}^{1} = 1$ and $\PP(M_i^1 = 1) = \pi_i^1$;
  $M_2$ is as well a multinomial variable with parameter
  $\pi^{2} = (\pi^{2}_1, \ldots, \pi^{2}_{d})$; and the variables
  $\epsilon_i$, $i\le p-d$ and $\zeta_j$, $j\le d$ are multiplicative
  noises, with different tail behaviour. Assume for
  simplicity that all $\epsilon_j$'s (\emph{resp.} $\zeta_j$'s) share
  the same survival function $S_{\epsilon}$ (\emph{resp.}
  $S_{\zeta}$) and that for all $s,t>0$,
  \begin{equation}
    \label{eq:differentTails}
    \lim_{y\to \infty} S_{\epsilon}(s^{-1}y) / S_{\zeta}(t^{-1}y) = 0. 
  \end{equation}
  Condition~\eqref{eq:differentTails} is satisfied \emph{e.g.} with
  Pareto noises,
  $S_{\epsilon}(y) = y^{-\alpha_1}, S_{\zeta}(y) = y^{-\alpha_2}$ with
  $\alpha_1>\alpha_2> 0$, or with Exponential versus Pareto noises,
  $S_{\epsilon}(y) = e^{-\alpha_1 y}, S_{\zeta}(y) = y^{-\alpha_2}$,
  $\alpha_1,\alpha_2>0$.  The random vectors
  $M^{1}, M^{2}, \epsilon, \zeta, V,W $ are independent.
  Finally  the covariate vectors $V$ and $W$ are made of independent   components $V_i, W_j$, with nonnegative, bounded support included in an interval $[a,b]$ with $0\le a<b<\infty$.

  In this generic example, $Y_1$ has a lighter tail than $Y_2$, so
  that it is the main risk factor regarding large values of $Y$, and
  it is intuitively desirable for a formal definition of tail
  conditional independence to be such that $Y$ is tail conditionally
  independent from $V$ given $W$ here.

  We now consider two special cases regarding the 
  marginal distributions of the covariates $V_j,W_j$.  recall that $[a,b]$ contains the support of each $V_i$ and each $W_j$.
 \begin{enumerate}
  
 \item[(i)] As a first go assume that  $a>0$.  
 Then   both \TCI and \TCIG hold. The proof is deferred to the supplementary material,  Section~\ref{sec:Example_bothTCI-TCIG}.

 \item[(ii)] Assume now that $a=0$, more specifically that
    each variable $V_j,W_j$ follows a binary Bernoulli distribution with
parameter $\tau \in (0,1)$ (the choice of a common $\tau$ merely
simplifies the notations). In Section~\ref{sec:TirexNotImpliesGardes} from the supplementary material we show that \TCIG does not hold, while \TCI does.

\end{enumerate}

Notice that   the difference between the two cases concerns only the marginal distribution of the covariate,   namely whether
  $\PP[W_j=0]>0$ is key. This seemingly minor variation results in
  fact in potential failure of \TCIG, while \TCI remains
  true.  The main conclusions of our comparison   between the two definitions (\TCI and \TCIG) in the supplementary material, Section~\ref{sec:compareGardes_long},  may be summarized as follows. 
  \begin{enumerate}
  \item Neither condition implies the other in general, except for discrete covariates where \TCIG implies \TCI.

  \item \TCIG  criterion concerns the additional information brought
    by $V$ regarding the probability of the event $Y>y$, \emph{after}
    conditioning on $W$. The criterion is satisfied if the additional
    information  is negligible, for \emph{all} possible values $W=w$,
    even those values such that the conditional distribution of $Y$ given
    $W=w$ is shorter tailed than the marginal distribution of $Y$.
    Indeed \TCIG is primarily designed for quantile regression, and
    the focus is not on the tail of $Y$'s distribution, but instead on
    the tails of the conditional distributions of $Y$ given $W$. This
    is the informal reason why \TCIG is not satisfied in the example
    above, Case~(ii).

\item  In constrast \TCI is designed for prediction of extreme values of $Y$. It is an integrated version of \TCIG with respect to the variable $(V,W)$,  with a weight function 
  granting more importance to $w$'s such that the ratio $\PP[Y>y\given W=w] / \PP[Y>y]$ is  large as $y\to\supy$. In  words, 
  \TCI is comparatively more sensitive to values $w$ such that the
conditional probability given $W=w$ of an exceedance $Y>y$ is  large. 
  \end{enumerate}
  
\subsection{Technical consequences of \TCI, parallel with traditional conditional independence}\label{sec:consequences_weakformulations}
  
  Definition~\ref{def:tail_cond_indep} implies equivalent weak
  formulations of the traditional conditions (CI-1,CI-2,CI-3) reviewed
  in the background section. 
\begin{proposition}\label{prop:equi_tail_cond_exp}
  If $\tailindep{Y}{V}{W}$ in the sense of Definition~\ref{def:tail_cond_indep}, then  the following equivalent  conditions~\emph{(i), (ii), (iii)} 
  hold. 
  \begin{enumerate}
  \item[(i)] 
    For any real-valued functions $g$ and $h$, measurable and bounded, we have
   \[
     \frac{\EE\Big[ g(V) h(W)  \Big( \PP[Y>y  \given V,W]  - \PP[Y>y    \given  W] \Big)  \Big] } {\PP[Y>y]}  \xrightarrow[y\to \supy]{} 0.
    \]
\item[(ii) ] 
For any real-valued functions $ g$ and $h$, measurable and bounded, we have
 \begin{equation*}
   \frac{\EE\Big[ \, h (W) \Big( \, \EE[ \un\{Y>y \}  g(V) \given W] -  \EE[ \un\{Y>y \} \given W] \EE[g(V)   \given  W] \, \Big) \, \Big] }{\PP[Y>y]} \xrightarrow[y\to \supy]{} 0.
 \end{equation*}
\item[(iii)] 
  For any real-valued functions $ g$ and $h$, measurable and bounded, we have
   \[
\frac{\EE\Big[ \, h(W) \un\{Y>y \} \; \Big( \EE[g(V)\given Y,W]  - \EE[g(V)   \given W] \Big)  \,\Big] } {\PP[Y>y]}  \xrightarrow[y\to \supy]{} 0.
    \]
  \end{enumerate}
 \end{proposition}

\begin{remark}[Relevance of Proposition~\ref{prop:equi_tail_cond_exp} for our purpose]\label{rem:interpretation-tailCI}
  {
    From a technical perspective,
  Property~(iii) in Proposition~\ref{prop:equi_tail_cond_exp} is key to obtain  the tail analogues of the SIR and SAVE  principles (Theorems~\ref{th:SIR_extreme},~\ref{th:SAVEx} in  Section~\ref{sec:tail-SIR}).  This is not surprising insofar as the traditional condition (CI-3) for  conditional independence  in  Definition~\ref{def:conditional_indep} is central to prove the SIR/SAVE principles. 
}

Whether the converse implication from Proposition~\ref{prop:equi_tail_cond_exp} holds true in general, \ie whether Conditions~(i), (ii), (iii) imply  \TCI  remains an open question which is not directly relevant for our purposes and thus left for future works.
\end{remark}

\begin{proof}[Proof of Proposition~\ref{prop:equi_tail_cond_exp}]
  
  We first show the equivalence  (ii)$\Leftrightarrow$(iii) 
  by proving that the left-hand sides of the two conditions are identical. Indeed  if $g$ and $h$ are bounded and measurable, then 
\begin{align*}
\EE\Big[ h(W) \un\{Y>y \}   \EE[g(V) \given Y,W ] \Big]=&\EE\Big[ h(W) \un\{Y>y \} g(V) \Big] \\
  = & \EE\Big[ h(W)  \EE[ \un\{Y>y \}  g(V) \given W ] \Big] \,,
\end{align*}
while
\begin{align*}
\EE\Big[ h(W) \un\{Y>y \} ( \EE[g(V) \given W] )\Big]=&\EE\Big[ h(W)\EE[\un\{Y>y \} \given W] \EE[g(V)   \given  W]  \Big].
\end{align*}

%
To show that 
(ii)$\Rightarrow$(i), 
note that
\begin{align*}
\EE\Big[ g(V) h(W)  \EE[ \un\{Y>y \} \given V,W ]\Big]& =  \EE\Big[ g(V)h(W)   \un\{Y>y \}  \Big]\\
&  = \EE\Big[ h(W)  \EE[ g(V)   \un\{Y>y \} \given W ] \Big]\\
&  = \EE\Big[  h(W)\EE[ g(V) \given W] \EE[  \un\{Y>y \} \given W  ]\Big] + r_1(y)\\
& =  \EE\Big[ g(V)  h(W)   \EE[  \un\{Y>y \} \given W ]  \Big]+ r_1(y),
\end{align*}
 where $\lim_{y\to \supy} r_1(y) / \PP[Y>y   ]= 0$ by Condition (ii). 

 The argument for the converse implication (ii)$\Leftarrow$(i)
 is similar: 
 \begin{align*}
 \EE[h (W)  \un\{Y>y \} g(V)  ] &=    \EE\Big[h (W)  \EE[\un\{Y>y \}\given V,W] g(V) \Big] \\
 &= \EE\Big[h (W)  \EE[\un\{Y>y \}\given W] g(V) \Big] +  r_2(y), 
\end{align*}
where  $\lim_{y\to \supy} r_2(y) / \PP[Y>y] = 0$ under condition (i). 

%
Finally we show that Property~(i) 
from Proposition~\ref{prop:equi_tail_cond_exp} is satisfied under the \TCI assumption from  Definition~\ref{def:tail_cond_indep}. Let $g,h$ be bounded, measurable functions defined on $\mathcal{V},\mathcal{W}$ respectively and let $\|g\|_\infty$ and $\|h\|_\infty$ denote their supremum norm. 
  By Jensen's inequality,
  \begin{align*}
    &  
      \PP[Y>y]^{-1} \left|\EE\Big[ g(V) h(W)  \Big( \PP[Y>y  \given V,W]  - \PP[ Y>y    \given  W] \Big)  \Big] \right| 
    \\
    &\le \|g\|_\infty\|h\|_\infty \PP[Y>y]^{-1} 
      \mathbb{E} \;\Big| \PP[Y>y  \given V,W]  -
      \PP[ Y>y    \given  W] \Big| \,, 
  \end{align*}
  where the right hand side tends to zero under
  Condition~\eqref{eq:newdefTCI} from
  Definition~\ref{def:tail_cond_indep}. 

\end{proof}
\subsection{Extreme dimension reduction spaces}\label{sec:extremeDRS}
In  the context of statistical regression,  we now define  extreme sufficient dimension reduction  subspaces in a similar fashion to the usual SDR spaces. 
\begin{definition}[Extreme SDR space and extreme central space]\label{as:tailCI}~
  \begin{itemize}
  \item
    An extreme SDR space for the pair $(Z,Y)$ is a subspace $E_e$ of $\rset^p$ such that
    $\tailindep{Y}{Z}{P_eZ}$, 
    where $P_e$ is the orthogonal projection on $E_e$.  In other words $E_e$ is called an extreme SDR space whenever 
    \begin{align}\label{eq:tci_excessY}
\EE\left|\frac{\mathbb P ( Y>y | Z ) - \mathbb P (Y>y | P_eZ )}{\PP[Y>y]}  \right|&\xrightarrow[y\to \supy]{} 0. 
\end{align}

\item An  extreme central space  $E_{e,c}$
  for the pair $(Z,Y)$ is an extreme SDR space of minimum dimension. 
  \end{itemize}
\end{definition}
Investigating sufficient conditions ensuring uniqueness of an extreme
central space is left for future studies. Instead, in the present
paper we shall consider \emph{an} extreme SDR space $E_e$ and we shall
show that under appropriate assumptions, inverse extreme regression
objects, namely limits of conditional expectations $\EE[Z\given Y>y]$
(Theorem~\ref{th:SIR_extreme}) and second order variants
(Theorem~\ref{th:SAVEx}) belong to $E_e$.  In particular they belong
to   any extreme \emph{central} space.  

\begin{remark}[Relationship between the central space and its extreme counterpart]\label{rem:inclusionSDR-ExSDR}
  Because Equation~\eqref{eq:tci_excessY} holds true for any
  $y\in \mathbb R$ when $E_e$ is chosen as a (non extreme) SDR space
  for the pair $(Z,Y)$, any SDR space for $(Z,Y)$ is an extreme SDR
  space. Thus, upon uniqueness of the central space $E_c$ and the
  extreme central space $E_{e,c}$, it holds that
  $ E_{e,c} \subset E_c$. Examples of other dimension reduction
  subspaces more specific than $E_c$ but not related to the extreme
  value of $Y$ include the central mean subspace
  \citep{cook2002dimension} and the central quantile subspace
  \cite{christou2020central}.
  \end{remark}


\section{Tail Inverse Regression}\label{sec:tail-SIR}
In the sequel, we consider an extreme SDR space $E_e \subset \rset^p$
for the pair $(Z,Y)$ in the sense of Definition~\ref{as:tailCI}.  That
is, we assume that $\tailindep{Y}{Z}{P_e Z}$ as in
Definition~\ref{def:tail_cond_indep}, where $P_e$ is the orthogonal
projection on $E_e$. Also we define $Q_e = I_p - P_e$.  In order to
adapt the SIR strategy to this tail conditional independence
framework, we show the following result which is a `tail version' of
the SIR principle (Proposition~\ref{prop:sir_principle}).  In the
remainder of this paper let $\|\point\|$ denote any norm on a finite
dimensional vector space.
\begin{theorem}[TIREX1 principle]\label{th:SIR_extreme}
  Assume the following conditions regarding  the pair $(Z,Y)$ and the extreme SDR space $E_e$.
  \begin{enumerate}
  \item \ (Uniform integrability): \\
   The random variables  $g_{1,A}(Z) =  \|Z\| \un\{\|Z\| >A\}  $,  $g_{2,A}(Z) = \EE[ \|Z\| \un\{\|Z\| >A\} \given P_eZ]$  indexed by $A\in\rset$ satisfy 
\begin{align}
&\lim_{A\to \infty }\limsup_{y\to y^+} \EE[ g_{k,A}(Z) \given Y>y ]  = 0 ,\qquad k = 1,2 \;;\label{eq:unif_integrabilbity}
\end{align}
\item \ (LC) The standardized vector $Z$ satisfies the linearity condition~\eqref{eq:LC} relative to~$P_e$;
  \item \ (Convergence of conditional expectations)  For some $\ell\in\rset^p $,  
  \begin{equation}
    \label{eq:limitEZ-largeY}
      \EE[Z \given  Y>y ]\xrightarrow[y\to y^+]{}\ell. 
  \end{equation}
  \end{enumerate}

Then $\ell \in E_e$.
\end{theorem}
\begin{proof} 
 We need to show that  $Q_e \ell = 0$.   By continuity of the projection operator $Q_e$ it is enough to show that $Q_e\EE[Z\given Y>y]\to 0$ as $y \to\supy$. On the other hand the linearity condition (LC)~\eqref{eq:LC} ensures that $Q_e\EE[Z\given P_eZ] = Q_e P_eZ = 0$ almost surely.
  Thus letting $p_y = \PP(Y>y)$ one may write
  \begin{align*}
    Q_e\EE[Z\given Y>y]& = p_y^{-1}Q_e\EE[ Z\un\{Y>y\} ] \\
    & = p_y^{-1}\EE\Big[\big(Q_e\EE[ Z|P_eZ, Y ]  - Q_e\EE[Z\given P_eZ]\big)  \un\{Y>y\}\Big], 
  \end{align*}
  because the second term of the difference inside the expectation of
  the second line is zero.  

Let $A>0$ and consider separately the case when $Z\leq A$ and $Z> A$, so that  
\begin{align*}
&Q_e \EE[ Z  \un\{Y>y\}] \\
  &=   Q_e  \EE\Big[ \big(  \EE[ Z\un\{\|Z\|\leq A\}  \given P_eZ,Y ]  -
    \EE[Z \un\{\|Z\|\leq A\} \given P_eZ]  \big) \un\{Y>y\}\Big]\\
  & \qquad + Q_e  \EE\Big[ \big(  \EE[ Z\un\{\|Z\|> A\}  \given P_eZ,Y]   -
    \EE[Z \un\{\|Z\|> A\}  \given P_eZ]  \big) \un\{Y>y\}\Big]. 
\end{align*}
For the first term of the above display, using Condition~(ii) 
of Proposition~\ref{prop:equi_tail_cond_exp} 
with $h = 1$ and $g(Z) = Z\un\{\|Z\|<A\}$, we obtain that 
\begin{align}\label{eq:<A}
  p_y^{-1}\, Q_e  \EE\Big[\big(   \EE[ Z\un\{\|Z\|\leq A\}  \given P_eZ,Y]  - \EE[Z\un\{\|Z\|\leq A\}  \given P_eZ] \big) \un\{Y>y\} \Big]
  \xrightarrow[y\to\supy]{}0.
\end{align}
For the second term corresponding to $Z> A$, we use that $\|Q_ez\| \leq \|z\|$, the Jensen inequality and the triangular inequality, which yields 
\begin{align*}
  &\Big\| Q_e  \EE\Big[ \big(  \EE[ Z\un\{\|Z\| > A\}  \given P_eZ,Y]  -
    \EE[Z\un\{\|Z\|> A\}  \given P_eZ] \big) \un\{Y>y\} \Big]\Big\|\\
  & \leq    \EE\big( \{ \EE[ \|Z\|\un\{\|Z\|> A\}  \given P_eZ,Y]  +
    \EE[\| Z\| \un\{\|Z\|> A\}  \given P_eZ]\}  \un\{Y>y\} \big)\\
    & = \EE[ g_{1,A}(Z) \un\{Y>y\} ]  +
    \EE[ g_{2,A}(Z) \un\{Y>y\}]
\end{align*}
By (\ref{eq:<A}) and the previous decomposition, we have shown that
\begin{align*}
&\limsup_{y\to\supy} \|Q_e \EE[ Z  |Y>y] \| \leq \limsup_{y\to\supy}  \EE[ g_{1,A}(Z)\given  Y>y ]
     + \limsup_{y\to\supy}
    \EE[ g_{2,A}(Z) | Y>y ]. 
\end{align*}
By further letting $A \to \infty$,   by Assumption~(\ref{eq:unif_integrabilbity}), the right-hand side is arbitrarily small. 
This shows that $\lim_{y\to \supy} Q_e \EE[ Z \given Y>y] = 0$ and the proof is complete.

\end{proof}

\begin{remark}[special case: Tail conditional distribution]\label{rem:tailDistrib_SIrex}
  A particular framework justifying the existence of the limit $\ell$
  (Condition~\eqref{eq:limitEZ-largeY} in the statement of Theorem~\ref{th:SIR_extreme}) is
  the following. Assume that the covariate $Z$ admits a \emph{tail
    conditional distribution} given $Y$, in the sense that the
  distribution of $Z$ conditional to $Y>y$ converges as
  $y\to\supy$. In other words assume that there is a probability
  distribution $\mu$, that we may call the \emph{tail conditional
    distribution} of $Z$ given $Y$, such that for all bounded,
  continuous function $g$ defined on $\rset^p$,
 \[
    \EE[g(Z)\given Y>y] \xrightarrow[y\to\supy]{} \mu(g) :=\int_{\rset^p} g \ud\mu.
  \]
  By virtue of proposition 2.20 in \cite{van_der_vaart:1998}, 
  if  the uniform integrability condition \eqref{eq:unif_integrabilbity} is satisfied regarding the functions $g_{1,A}$ and if $Z$ admits a tail conditional distribution $\mu$ relative to $Y$, then it holds that
  \begin{align*}
        \EE[Z \given  Y>y ]\xrightarrow[y\to y^+]{ } m_\mu: =  \int z d\mu(z), 
  \end{align*}
so that condition~\eqref{eq:limitEZ-largeY} automatically holds with $\ell = m_\mu$.  

\end{remark}

\begin{remark}[relationships with graphical models for  extremes]\label{rem:relationships-graphical}
  The above notion of tail conditional distribution reveals a connection between the present work and graphical modeling approaches in  EVT. Namely, assuming a tail conditional distribution of $Z$ given $Y$, and requiring in addition that the random variable $Y$ is regularly varying,  is equivalent to assuming \emph{one-component regular variation} of the pair $(Y,Z)$, a concept first introduced by \cite{hitz2016one}.  See  in particular their Theorem~1.4,  where the pair $(X,Y)$ plays the role of the pair $(Y,Z)$ in the present work. 

  The notion of conditional independence at extreme levels also plays an important role in \cite{engelke2020graphical}. However our  work  departs significantly from  the latter, in so far as    the general context in the cited reference is that of unsupervised learning. All  considered variables play a symmetric role --there is no target variable nor covariate --,  and  they rely on an assumption of joint multivariate regular variation of the considered random vector which is by no means necessary in our context.

\end{remark}

\begin{remark}[Special case: extreme central space]\label{rem:extremeCentralSpace}
  Upon uniqueness of the extreme central space $E_{e,c}$,  under the assumptions of Theorem~\ref{th:SIR_extreme} we obtain that $\ell\in E_{e,c}$.
\end{remark}


  \begin{remark}[Sufficient condition for uniform integrability]
Using the fact that for any $\varepsilon>0$,  $\un\{\|Z\|>A\} \le \|Z\|^{\varepsilon}/A^{\varepsilon}$, a sufficient condition for the uniform integrability condition~(\ref{eq:unif_integrabilbity}) is that 
  \begin{align*}
\limsup_{y\to \supy} \frac{\EE[\|Z\|^{1+\varepsilon} \; \un\{Y>y\}]}{\PP[Y>y]} <\infty,
\end{align*}
for some $\varepsilon>0$.
\end{remark}


A natural strategy in view of Theorem~\ref{th:SIR_extreme} is to consider empirical counterparts of the conditional expectations $\EE[Z\given Y>y]$ for large values of $y$ so as to estimate the limit value $\ell$, which belongs to any extreme SDR space. Asymptotic statistical guarantees for this approach are derived in Section~\ref{sec:estimationWithUniform}.  However  an obvious limitation of Theorem~\ref{th:SIR_extreme} 
is that 
it recovers a single direction within an extreme SDR space, namely the line $\{t \ell, t\in\rset \}$ in the case where $\ell\neq 0$. If a unique extreme central space exists and if this subspace is one dimensional,
then indeed the generated line and the extreme central space coincide. 
To consider situations where the minimum dimension of an extreme   SDR space is greater than one,  we develop an extreme analogue of the SAVE framework by considering conditional second order moments. The main result justifying this approach is encapsulated in Theorem~\ref{th:SAVEx} below. 

\begin{theorem}[TIREX2 principle]\label{th:SAVEx}
   Assume  $(Z,Y)$ and the extreme SDR space $E_e$ satisfy the assumptions of Theorem~\ref{th:SIR_extreme} and that in addition, 
  \begin{enumerate}
  \item \ (second order uniform integrability):
\begin{align}
&\lim_{A\to \infty }\limsup_{y\to y^+} \EE[ h_{k,A}(Z) \given Y>y ]  = 0 ,\qquad k = 1,2 \;,\label{eq:unif_integrabilbity_2}
\end{align}
where $h_{1,A}(Z) =  \|Z\|^2 \un\{\|Z\| >A\}  $ and $h_{2,A}(Z) = \EE[ \|Z\|^2 \un\{\|Z\| >A\} \given P_eZ]$ for $A\in\rset$;  
\item \ (CCV) The standardized vector $Z$ satisfies the constant variance condition~\eqref{eq:CCV}  relative to~$P_e$;
\item \ (Convergence of conditional expectations)  For some $S\in\rset^{p\times p} $,
  \begin{equation}
    \label{eq:limitEZ2-largeY}
      \EE\big[ZZ^\top  \, | \,  Y>y \big]\xrightarrow[y\to \supy]{} S + \ell \ell^\top. 
    \end{equation}

  \end{enumerate}
      Then $\Span(S - I_p) \subset E_e$, \ie  $Q_e (S-I_p) = 0$. 
  %
\end{theorem}
Notice that  the existence of $\ell = \lim \EE[ Z \given Y>y]$ is part of the assumptions of  Theorem~\ref{th:SIR_extreme}  and that in the latter framework,  $Q_e\ell\ell^\top = 0$. Thus condition~\eqref{eq:limitEZ2-largeY} is equivalent to  requiring that $\Var{Z\given Y>y}$ converges to some limit variance $S$ as $y\to \supy$ and the conclusion  can be rephrased as $Q_e(\, \EE\,\![ZZ^\top\,|\, Y>y\,] - I_p \,)\to 0$ as $y\to\supy$,  or equivalently $Q_e (\Var{ Z |Y>y} - I_p) \to  0$.
The technique of the proof is similar to that for  Theorem~\ref{th:SIR_extreme}. The key is to observe that the Constant Conditional Variance assumption allows to introduce a difference
$(\, \EE\,\![ZZ^\top \,|\, P_e Z, Y] - \EE\,\![ZZ^\top \,|\, P_e Z] \,) \un\{Y>y\}$ which is asymptotically negligible because of the \TCI  assumption.  The details are gathered in  the supplement material, Section~\ref{ap:proofTIREX2principle}. 
    

\section{Estimation}\label{sec:estimationWithUniform}

This section is devoted to the 
statistical implementation of our main results from
Section~\ref{sec:tail-SIR}.  
  Theorems~\ref{th:SIR_extreme} and~\ref{th:SAVEx} show that the quantities
  $\ell$ and $S$ 
  in the limits of the two statements are
  key to estimate the extreme SDR space, because $\ell\in E_e$ and
  $\Span(S - I_p)\subset E_e$. A natural first idea would be to use as
  an estimate an empirical version of the quantities
  $\EE\,\![Z \,|\, Y>y \, ]$ or $\EE\,\![ZZ^\top\,|\, Y>y\,]$ for a high
  threshold $y$ growing with the sample size $n$. A typical choice of such a threshold is the quantile of $Y$ at a probability level $1- k/n$,  where $k=k(n)$ is an intermediate sequence
  such that $k(n) \to \infty$ and $k(n)/n\to 0$ as $n\to\infty$. Here
  we propose a refinement of this strategy  integrating out
  the latter quantities over varying quantiles at probability levels $1- uk/n$ for
  $u\in(0,1)$. Such a refinement follows the proven approaches based on the CUME and CUVE matrices described in the background
  section~\ref{sec:background-SIR}. For this purpose we introduce and prove the asymptotic normality of the empirical processes associated with the first and second order method, that are respectively the specialisation of the SIR/CUME and the SAVE/CUVE processes to extreme regions of the target $Y$.  

Even though the first order method is potentially less fruitful than the second order one since the limit $\ell$ in Theorem~\ref{th:SIR_extreme} is a single vector, it helps build the intuition about the statistical theory for both the first order and second order methods. In addition, the first order method turns out to be more stable in some of our experiments.

Some notations are introduced in Section~\ref{sec:notations_estim}.  We provide asymptotic theory for the first and second order empirical processes in Section~\ref{sec:mainresult-empirical}. Section~\ref{sec:method} summarizes the methods we suggest for estimating $E_e$.

\subsection{Framework and notations}\label{sec:notations_estim}





For any right-continuous cumulative distribution function $H$ (be it empirical or not), we shall denote by $H^-$ the  left-continuous inverse of $H$,
$ H^- (u) = \inf\{ x\in \mathbb R \,:\, H(x) \geq u\}. $
Recall that with these conventions, for $u\in [0,1]$ and $x \in \rset$, we have
\begin{equation}
  \label{eq:H_inverse}
  H(x) \ge  u \iff x \ge H^-(u).
\end{equation}
  For any i.i.d. sample $(T_i)_{i \le n}$ associated with a real random variable $T$ with cumulative distribution $H$, we use the standard definition of the empirical distribution function, 
\begin{align}\label{eq:def_c0}
\hat H(x) = n^{-1} \sum_{i=1} ^n \un\{ T_i \leq x\}. 
\end{align}

For notational and mathematical convenience
we shall work with the negative target $\tilde Y = - Y$ 
and denote  the \cdf of $\tilde Y$ as $F$ ,
which we assume to be continuous in the remainder of this paper.
For simplicity we write $k$ instead of $k(n)$ for the intermediate sequence defined at the beginning of this section, as is customary in extreme value statistics. 
Consider the first order and second order inverse regression functions
 $C_n(u), B_n(u)$, 
  \begin{align}
    C_n(u) & =  \frac{n}{k}\EE[Z\un\{ \tilde Y <  F^-(u k /n )\}] \;, \label{eq:Cndeterministic} \\
      B_n(u) &= \frac{n}{k}\EE[( ZZ^\top - I_p)\un\{ \tilde Y  <  F^-(uk/n)\} ] .      \label{eq:Bndeterministic}       
  \end{align}
The empirical versions of~\eqref{eq:Cndeterministic} and~\eqref{eq:Bndeterministic}  based on an independent sample
  $(Z_i, Y_i)$ identically distributed as the pair $(Z,Y)$ are
\begin{align}
  \hat C_n(u)
  &  = \frac{1}{k} \sum_{i=1}^n Z_i\un\{\tilde Y_i \le \hat F^-(uk/n) \}\label{eq:hatCn} \,,\\
  \hat B_n(u) 
  & = \frac{1}{k} \sum_{i=1}^n (Z_iZ_i^\top - I_p) \un\{\tilde Y_i \le \hat F^-(uk/n) \} .\label{eq:hatCn2}
\end{align}
  Extensions to the more realistic situation where the pair $(X,Y)$ is observed with the mean and covariance of $X$ unknown are gathered in Section~\ref{sec:nonstandard} from the supplementary material. 
  



\subsection{Main result}\label{sec:mainresult-empirical}
 The remainder of this section aims at establishing the
weak
convergence of the (tail) empirical processes associated with 
TIREX, respectively $\sqrt{k}(\hat C_n(u) -C_n(u))$ and $\sqrt{k} (\hat B_n(u) - B_n(u))$ in the space of bounded functions
$\ell^\infty([0,1])$. This is achieved in Corollary~\ref{cor:weakCV-SIREX-SAVEX}.

  A key point of our analysis,  which follows from the continuity of $F$, is that the functions $C_n(u),B_n(u)$ and their estimates $\hat C_n(u), \hat B_n(u)$ are invariant under the transformation
  $U = F(\tilde Y)$. 
  More precisely, with the latter  notation, we have the following identities
  \[
    C_n(u) = \frac{n}{k}\EE\big[\, Z\un\{U < uk/n\} \,\big]\;, B_n(u) =
    \frac{n}{k}\EE \big [\, (ZZ^\top -I_p)\un\{U < uk/n\} \, ],
  \] and the
  processes $\hat C_n(u), \hat B_n(u)$ remain the same when
  constructed from the initial sample $(X_i,\tilde Y_i)$ or when
  constructed from the uniformized sample $ (X_i,U_i)$.  Indeed 
  for $u\in [0,1]$, it holds that
\[
  \un\{\tilde Y_i \le \hat F^-(uk/n) \}  = 
    \un\{U_i \le \hat F_U^-(uk/n) \},\qquad  \text{a.s.}, 
  \]
  where 
  $\hat F_U$ is the empirical distribution function associated with the uniform sample $U_1,\ldots, U_n$, see Fact~\ref{fact:cdf} in the supplementary material  for a short proof. 
 These facts are  a known feature of rank based estimators; see for instance \cite{fermanian+r+w:2004} in the copula estimation context and \cite{portier:2016} in the standard SIR context.

We now state our main result which is formulated in terms of a generic random pair $(V,Y)$,  an \iid\ sample thereof $(V_i,Y_i), i\le n$, and a measurable function $h: \rset^r \to \rset^q$, where $Y$ is the response variable as above,  
the  covariate  $V$ is a random vector of size $r\in\nset^*$ and $h$ is such that the random vector $h(V)$ has  finite second moments.   Define 
\begin{align*}
  D_n(u) &= \frac{n}{k}\EE[h(V) \un\{ \tilde Y  <  F^-(uk/n)  \}],  \\
  \hat D_n(u)= & \frac{1}{k}\sum_{i=1}^n h(V_i)\un\{\tilde Y_i  \le \hat F^-(uk/n) \}. 
\end{align*}
Weak convergence of the TIREX1 and TIREX2 processes (Corollary~\ref{cor:weakCV-SIREX-SAVEX}) is obtained  upon setting  $V = Z$ and respectively     $h_C(z) = z$ and $h_B(z)  = \Vector(zz^\top - I_p)$, where for any matrix $M \in \rset^{r\times s}$, $\Vector(M)$ denotes the vector of size
$r\times s$ obtain by concatenating the columns of $M$. 

\begin{theorem}[Tail empirical process for a generic pair $(V, Y)$ ]\label{thm:generic-pair-tailprocess}
 Suppose that the distribution  function $F$ of $\tilde Y = -Y$  is continuous and that, letting $U = F(\tilde Y)$, it holds that 
  \begin{enumerate}
  \item   for any $j \in\{ 1,\ldots, q\}$, the functions
  $u\mapsto \EE[ h(V)_j \un\{U\leq u \} ]$ and $u\mapsto \EE[ h(V)_j^2 \un\{U \leq u \} ]$ are differentiable on $(0,1)$ with a continuous derivative at $0$, 
\item  for all $M\geq 0$,
  $S(M) := \lim_{\delta \to 0} \EE[ h(V)h(V)^\top  \un\{\|V\|\geq M\} \given U \leq  \delta   ] $  exists  and is such that  $\lim_{M\to \infty} S(M)  = 0$,
  \item as $\delta\to 0$, $\EE[h(V)\given   U \le \delta ] $ converges to a limit $\nu\in\rset^q$. 
  \end{enumerate}
 Then  we have as $n\to\infty,k\to\infty,k/n\to 0$, 
\begin{align*}
\left\{ \sqrt k (\widehat D_n(u) - D_n (u) )\right\}_{u\in[0,1]} \leadsto \left\{ W (u) \right\}_{u\in[0,1]} ,
\end{align*}
where $W$ is a Gaussian process with mean zero and covariance function
\begin{align}
  (s,t) \mapsto & s\wedge t \; \big( \Xi - \nu\nu^{\top} \big)\,,  \label{eq:limitCovarianceGeneric}
\end{align}
with $\nu$ as in the $3^{text{rd}}$ Condition of the statement and 
\begin{equation}
  \label{eq:defXi-limit2ndMom}
  \Xi = S(0) =  \lim_{\delta\to 0} \EE[h(V)h(V)^\top \given U \le \delta]  \in \rset^{q\times q} .
\end{equation}
  
\end{theorem}
\begin{cor}[Weak convergence of the TIREX1  and TIREX2 processes]\label{cor:weakCV-SIREX-SAVEX}
  By choosing the pair $(V,Y) =(Z,Y)$ and assuming that the function $h_C(z) = z$  (\emph{resp.} $h_B(z) = \Vector(zz^\top  -I_p)$)  satisfies the assumptions of Theorem~\ref{thm:generic-pair-tailprocess}, the TIREX1 process $\sqrt{k}(\hat C_n(u) - C_n(u))$
  (\emph{resp.} the TIREX2 process $\sqrt{k}(\hat B_n(u) - B_n(u)$) converges weakly in $\ell^\infty(0,1)$ to a  tight Gaussian process $W_C$ (\emph{resp.} $W_B$) with covariance function given by \eqref{eq:limitCovarianceGeneric} with $V = Z$ and $h = h_C$ (\emph{resp. } $h = h_B$)
\end{cor}


\begin{proof}[Proof of Theorem~\ref{thm:generic-pair-tailprocess}]
  
Consider  the pseudo-empirical version of $D_n(u)$,  
\begin{align}\label{eq:tildeC}
   \widetilde D_n ( u)  =k^{-1} \sum_{i=1}^n h(V_i)\un\{ U_i \leq u k/n \}
     =k^{-1} \sum_{i=1}^n h(V_i)\un\{ \tilde Y_i \leq F^-( u k/n) \} \,.
\end{align}
Notice that $\widetilde D_n$ is not observed but serves as an intermediate quantity through the following key identity:
\begin{align*}
  \hat D_{n }(u) =\widetilde D_n \Big( \frac{n}{k} \hat F_U^-( u k/n ) \Big), 
\end{align*}
where $\hat F_U$ is the empirical \cdf associated with the sample $(U_i, i\le n)$. 
Introducing the process
\begin{equation}
  \label{eq:intermediateProcesses}
\begin{aligned}
\widetilde \Gamma(u) &= {\sqrt k} \left( \widetilde D_n(  u )   -  D_n( u )    \right),\quad  u \in [0,1],  
\end{aligned}
\end{equation}
we have the following decomposition
\begin{align}
\sqrt k (\hat D_{n }(u) - D_n( u) ) &= \widetilde \Gamma\Big(\frac{n}{k}\, \hat F^-\big(u k / n \big)\Big)   
                              +\sqrt k \Big( D_n \Big( \frac{n}{k}\, \hat  F_U^-\big( u  k / n  \big)\Big) -  D_n (u) \Big)\,. \label{decomp_as_normbis}
\end{align}
In the remainder of the proof,  
we show that the first term can be
replaced by $\widetilde \Gamma(u)$, while the second term can be
replaced by $-\nu\hat\gamma_1(u)$ 
  where $\hat\gamma_1$ is the tail empirical process for uniform
  random variables,
\begin{equation}
  \label{eq:gamma1}
  \begin{aligned}
    &\hat \gamma_1(u)   =  \sqrt k \left( \frac{n}{k}\hat F_U(uk/n)  - u\right) . 
  \end{aligned}
\end{equation}

Finally we  show that  the process
$(\hat\gamma_1(u), \widetilde \Gamma(u))_{u\in[0,1]}$ converges
jointly to a Gaussian process. 

\subsubsection*{ Intermediate results, uniform tail processes}
The main tools that we use in our proof of Theorem~\ref{thm:generic-pair-tailprocess} concern 
the weak convergence of  the  tail empirical (quantile) process associated with a uniform response variables. Many approaches have been considered to handle the behavior of such processes, see~\cite{csorgo+c+h+m:1986} for general empirical  processes and~\cite{einmahl1988strong} for the tail version.
For the sake of completeness we provide in the supplementary material (Section~\ref{sec:proof_genericLemma_1})  a different, direct  proof of Lemma~\ref{lemma:useful_lemma_generic} below, relying on `classes of function  changing with $n$' (\cite{vandervaart+w:1996})

\begin{lemma}\label{lemma:useful_lemma_generic}

  Under  the assumptions of 
  Theorem~\ref{thm:generic-pair-tailprocess},
  the process $\widetilde \Gamma$ defined in~\eqref{eq:intermediateProcesses} converges weakly in $\ell^\infty(0,1)$  to a tight Gaussian process $\widetilde W$ with covariance function 

\begin{align*}
(u_1,u_2) \mapsto    (u_1\wedge u_2 ) \Xi \,,
\end{align*}
where $\Xi$ is  defined in~\eqref{eq:defXi-limit2ndMom}  

\end{lemma}
An immediate consequence of   Lemma~\ref{lemma:useful_lemma_generic}, obtained upon setting $V = Z$ and $h(V) = 1$,  is the weak convergence of  the  tail empirical process for uniform random variables introduced in \ref{eq:gamma1}.  

\begin{cor}\label{lemma:useful_cor_1}
As $n\to \infty$, $ k\to \infty$,  and $k/n\to 0$, the uniform tail empirical process~\eqref{eq:gamma1}
weakly converges to  a standard Brownian motion $W_1$. 
\end{cor}

Combining Corollary \ref{lemma:useful_cor_1} and an appropriate variant of Vervaat's lemma (see Section~\ref{sec:vervaat} from the supplementary material) we obtain in Section~\ref{sec:proofUsefulLemma2} from the same supplement,  the following result. 
\begin{lemma}\label{lemma:useful_lemma_2}
As $n\to \infty$, $ k\to \infty$,
\begin{align*}
 \sup_{u\in (0,1]}  \left|  \sqrt k \left( \frac{n}{k}\hat F_U^-(u  k / n  )  - u \right)+  \hat\gamma_1(u) 
\right|  = o_{{\mathbb P}}(1).
\end{align*}

\end{lemma}

\subsubsection*{Separate and joint convergence  in  Decomposition~(\ref{decomp_as_normbis})}
We now show the following three relations: as $n\to\infty$,
\begin{align}
  \label{eq:unifContinGamma1}
&\sup_{u \in (0,1]}\Big|\widetilde \Gamma\Big( \frac{n}{k} \; \hat F_U^-\big( u  k / n  \big) \Big) - \widetilde \Gamma(u)\Big| = o_{\PP}(1),\\
&\label{eq:unifContinD}\sup_{u \in (0,1]}\Big|\sqrt k \left(D_n \Big(\frac{n}{k}\; \hat F_U^-\big( u  k / n  \big) \Big) -  D_n (u) \right)+  \nu  \; \hat \gamma_1(u) \Big| = o_{\PP}(1),\\
&\label{eq:jointconvergence}\begin{pmatrix}
  \hat \gamma_{1}(u)\\  \widetilde \Gamma(u)
\end{pmatrix}
\leadsto  W'(u) ,\end{align}
where $W'$ is a centered Gaussian process on $(0,1]$ with covariance function
$(s,t) \mapsto s\wedge t  \; \Xi'$. Here $\Xi' = S'(0)$ is the limit second moment matrix from Lemma~\ref{lemma:useful_lemma_generic} applied to  $ h'(V) = (1, h(V))$. More specifically, with this choice of $h'$, we have
\[\Xi'  =
  \begin{pmatrix}
    1 & \nu^\top \\
    \nu  & \Xi  
  \end{pmatrix} \in\rset^{(q+1)\times (q+1)}
\]
where 
$\Xi = \lim_{\delta\to 0} \EE\,[ h(V)h(V)^\top \, | \,  U\le \delta\, ] $. 

We first prove \eqref{eq:unifContinGamma1}. From Lemma~\ref{lemma:useful_lemma_generic}, 
the process $\widetilde\Gamma$ is tight, whence asymptotically equi-continuous, meaning that 
  \[
    \lim_{\delta\downarrow 0}\,\limsup_n\,\PP\,\Big(\,\sup_{|s-t|\le \delta} \,| \widetilde \Gamma(s) - \widetilde \Gamma(t) |\,\, >\epsilon  \, \Big) =0 . 
  \]
  Also, from  Lemma~\ref{lemma:useful_lemma_2} and Corollary~\ref{lemma:useful_cor_1},  $\sup_{u\in ( 0,1]} |(n/k) \hat F_U^-( u k/n )  -u|   = o_{\PP}(1)$. Combining the two yields \eqref{eq:unifContinGamma1}.
  
To prove \eqref{eq:unifContinD}, we apply the mean value theorem to get that
 \begin{align*}
   &  \sqrt k \left\{D_n \Big(\frac{n}{k} \hat F_U^-\big( u  k / n  \big) \Big) -  D_n (u) \Big)\right\} \\
   &= \frac{n}{\sqrt k} \bigg\{\EE\Big[ h(V) \un \{U\leq u_n \} \Big]_{ u_n =  \hat F_U^-( u k/n  )} - \EE \Big[ h(V) \un \{U\leq uk/n \} \Big]\bigg\}\\
 & = \frac{n}{\sqrt k}   \tilde g(  \tilde  U_{u,n } )  \Big\{ \hat F_U^-\big( u  k / n  \big) - u  k / n  \Big\}\\
 & =  \sqrt k  \, \tilde g(  \tilde  U_{u,n } )  \Big\{ \frac{n}{k} \;  \hat F_U^-\big( u  k / n  \big) - u  \Big\}\,,
  \end{align*}
  where $\tilde g(x) $ is the derivative of
  $ x\mapsto \EE[ h(V) \un\{U\leq x \} ]$ at point $x$ and
  $\tilde U_{u,n } $ lies on the line segment between
  $ \hat F_U^-( u k/n )$ and $uk/n$. Lemma~\ref{lemma:useful_lemma_2}
  and Corollary~\ref{lemma:useful_cor_1} imply that
  $\tilde U_{u,n}\to 0$ in probability uniformly over $u\in[0,1]$,
  thus by continuity of $\tilde g$ at $0$,
  $g(\tilde U_{u,k}) = \tilde g(0) + o_{\PP}(1)$.  We can further
  calculate $\tilde g(0)$ based on assumption 3 in
  Theorem~\ref{thm:generic-pair-tailprocess} as follows,
  \[\tilde g(0) = \lim_{u\to 0} \EE[h(V) \un\{U\le u\}] / u = \lim_{u\to 0} \EE[h(V)\given U\le u] = \nu. \]   
  Therefore, the relation \eqref{eq:unifContinD} is proved by applying
  Lemma \ref{lemma:useful_lemma_2}, and the Slutsky's lemma.  Finally,
  \eqref{eq:jointconvergence} follows from applying Lemma
  \ref{lemma:useful_lemma_generic} to the function
  $h'(V) = (1, h(V) )$.

  \subsubsection*{Conclusion}
By combining the decomposition in \eqref{decomp_as_normbis} with the relations \eqref{eq:unifContinGamma1}-\eqref{eq:jointconvergence}, we obtain that, as $n\to\infty$,
\[
\left\{\sqrt{k}\left(\widehat D_n (u) - D_n(u)\right) \right\}_{u\in [0,1]}\leadsto W := ( -   \nu  , I_q )\;  W'
  \]
  which is a Gaussian process with covariance function
  \[
    \begin{aligned}
 \Sigma(s,t) & = s\wedge t  \; ( -   \nu   , I_q )   \begin{pmatrix}
    1 & \nu^\top \\
    \nu & \Xi 
  \end{pmatrix}
  \begin{pmatrix}
    - \nu^\top \\ I_q 
  \end{pmatrix}  \\
  & = s\wedge t  \, \big(\Xi - \nu\nu^\top \big).
 \end{aligned}
    \]

\end{proof}
\subsection{Proposed estimation method}\label{sec:method}
This section summarizes the main steps of the first and second order methods that we propose  based on the processes $\hat C_n$ and $\hat B_n$.
We first  introduce TIREX1 and TIREX2 matrices in  parallel with 
the matrix $M_{\text{CUME}}$ defined in~\eqref{eq:M-cume} 
in our framework, following the integral based methods proposed by~\cite{zhu+z+f:2010}, see also \cite{portier:2016}. 
In line with the CUME~\eqref{eq:M-cume} 
matrix,  we define 
\begin{equation}\label{eq:CUMEX-CUVEX}
\begin{aligned}
  M_{\text{TIREX1}} & = \int_0^1C_n(u )C_n(u)^\top  \ud u \,,  \\ 
    M_{\text{TIREX2}} & = \int_0^1B_n(u )B_n(u)^\top  \ud u  \,,
  \end{aligned}
\end{equation}
where 
 $C_n$ and $B_n$ are defined in ~\eqref{eq:Cndeterministic} and ~\eqref{eq:Bndeterministic} respectively. We omit the dependency of the matrices on $n,k$ for convenience. An easy but important observation which underlies our strategy for estimating an extreme SDR space is the following lemma. 
\begin{lemma}[Consistency of the TIREX matrices]\label{lem:consistentMTIREX} \ \\
\noindent \emph{(i)} Under the assumptions of Theorem~\ref{th:SIR_extreme},
  \[
 M_{\textrm{TIREX1}} \longrightarrow \frac{1}{3}\ell\ell^\top \quad \text{ as } n\to \infty. 
\]
\noindent\emph{(ii)} Under the assumptions of Theorem~\ref{th:SAVEx},
  \[
 M_{\textrm{TIREX2}} \longrightarrow \frac{1}{3}( S  - I_p + \ell\ell^\top)^2\quad \text{ as } n\to \infty. 
\]
\end{lemma}
\begin{proof}
   Under the assumptions of the first statement,   for fixed $u$,  $C_n(u)C_n(u)^\top \to u^2 \ell\ell^\top$ as $n\to\infty$. The result follows by dominated convergence on $(0,1)$, which applies by virtue of Condition~\eqref{eq:unif_integrabilbity}. Indeed this uniform integrability assumption ensures that  for some constant $A>0$,  for $n$ large enough, for all $u\in(0,1)$,
  \begin{align*}
\|    C_n(u) \| & = \Big\| u\EE\big[Z \,|\, \tilde Y < F^-(uk/n) \,\big]\,\Big\| \\
    &\le u(A + \EE[ \|Z\|\un\{\|Z\|>A\} \,|\, \tilde Y < F^-(uk/n) ] \le u(A+1).  
  \end{align*}
The argument for the second statement is similar, 
up to a call to Condition~\eqref{eq:unif_integrabilbity_2} instead of~\eqref{eq:unif_integrabilbity}.
\end{proof}
As a consequence of Lemma~\ref{lem:consistentMTIREX}, both column spaces of $M_{\text{TIREX1}} $  and  $ M_{\text{TIREX2}} $ are asymptotically included in $E_e$. The column space of $M_{\text{TIREX1}} $ has dimension one while  that of $ M_{\text{TIREX2}} $ can be of any dimension not higher than that of $E_e$.
%
%
We propose the following estimation procedures based respectively on the processes $\hat C_n$ and $\hat B_n$. 

\subsubsection*{TIREX1}
\begin{enumerate}
  \item Choose $k \ll n$ and  $1 \leq d\leq p $.
\item Compute the estimated TIREX1 matrix, $\widehat M_{\text{TIREX1}} = \int_0^1 \hat C_n(u) \hat C_n^\top(u) \ud u $ using the identity given in \eqref{hat_cumex}. 
\item Perform an eigen decomposition of $\widehat M_{\text{TIREX1}}$ and keep the first $d$ eigenvectors  $(e_i, i\le d)$. 
\item ouptut: $\hat E_e = \Span(\{e_i, i\le d\})$.
\end{enumerate}
Choosing $d>  1$ is not immediately justified because the limit of  $M_{\text{TIREX1}} $ is a rank one matrix $\ell\ell^\top  / 3$ as indicated in Lemma \ref{lem:consistentMTIREX}. However, 
empirical evidence suggests that allowing $d>1$ can be useful to recover more components among the extreme central subspace basis. This is why we include this option in the algorithm.



\subsubsection*{TIREX2}
\begin{enumerate}
  \item Choose $k \ll n$ and  $1 \leq d\leq p $.
\item Compute the estimated TIREX2 matrix, $\widehat M_{\text{TIREX2}} = \int_0^1 \hat B_n(u) \hat B_n^\top(u) \ud u $ using the identity given in \eqref{hat_cuvex}. 
\item Perform an eigen decomposition of $\widehat M_{\text{TIREX2}}$ and keep the first $d$ eigenvectors  $(e_i, i\le d)$ associated with the highest eigen values. 
\item ouptut: $\hat E_e = \Span(\{e_i, i\le d\})$
\end{enumerate}

 We make the following remarks regarding the relationships between our main theoretical result Corollary~\ref{cor:weakCV-SIREX-SAVEX} and the proposed estimation methods TIREX1 and TIREX2. 

 \begin{remark}[Asymptotic normality of the TIREX matrices]\label{rem:normalityTirexMatrices}
  The asymptotic normality of the random matrices $\sqrt{k}(\widehat{M}_{\text{TIREX1}} - {M}_{\text{TIREX1}}  )$ and $\sqrt{k}(\widehat{M}_{\text{TIREX2}} - {M}_{\text{TIREX2}}  )$ could be obtained as a further consequence of Corollary~\ref{cor:weakCV-SIREX-SAVEX} with straightforward calculations.  
  This can be achieved by using the Delta-method as in the proof of~\cite{portier:2016}, Proposition~5. For the sake of conciseness  we leave  the detailed proof to interested readers. 
\end{remark}

\begin{remark}[Bias term]\label{rem:bias}
  Notice that the TIREX matrices $M_{\text{TIREX}}$ are deterministic but subasymptotic quantities which depend on the the choice of the ratio $k/n$. The ultimate goal in view of Lemma~\ref{lem:consistentMTIREX}  would be to obtain the limit distribution of $\sqrt{k}(\widehat{M}_{\text{TIREX1}} - \frac{1}{3} \ell\ell^\top )$
  and $\sqrt{k}(\widehat{M}_{\text{TIREX2}} - \frac{1}{3} (S - I_p + \ell\ell^\top) )$. An obvious way to do so is to assume that the bias terms   $\sqrt{k}({M}_{\text{TIREX1}} - \frac{1}{3} \ell\ell^\top )$ and   $\sqrt{k}({M}_{\text{TIREX2}} - \frac{1}{3} (S - I_p + \ell\ell^\top ) )$ converge to zero in probability, and use Slutsky's lemma. 
\end{remark}

\begin{remark}[Principal Component Analysis of the TIREX matrices]
  The  output of the TIREX methods is the eigen spaces of the
  estimated TIREX matrices. An important final step  is to show that such eigen spaces converges to the space spanned by the limits $1/3 \ell\ell^{\top}$ and $1/3 (S- I_p + \ell\ell^\top)$.
  {
    A possible starting point would be to use results from
    perturbation theory, see \emph{e.g.} \citep[Theorem
    3]{zwald2005convergence} where the Frobenius norm of the error is
    controlled by the inverse of a spectral gap.%
  }%
  Since this problem
  is left aside even in the traditional inverse regression literature
  we leave this question to further research while demonstrating the
  performance of the TIREX algorithms by numerical
  experiments. 
\end{remark}

\begin{remark}[Choices of $d,k$]\label{rem:choicekd}
  The choice of the intermediate sequence~$k$ of extreme order statistics is a standard issue in extreme value statistics. In our experiments (Section~\ref{sec:experiments}) we propose to choose $k$ by cross-validation. Theoretical investigation regarding this strategy is beyond the scope of this paper. Similarly, the choice of $d$ in the PCA decomposition of the matrix
  $\widehat M_{\text{TIREX2}}$ is a recurrent question in the PCA literature, which is also left to further research. In practice a natural and widely used strategy is an elbow method applied to the plot of the estimated  eigen values.
In the supervised learning context,
we recommend to choose $d$ by cross-validation. More generally (outside the supervised learning context), 
testing for the rank of the underlying matrix is a convenient method to infer the value of $d$. Such an approach has been succcessfully employed  in the SDR literature \citep{portier2014bootstrap} where the test statistics are usually based on the eigenvalues amplitude. Finally recall that  the limit of the matrix $M_{\text{TIREX1}}$ has rank one, so that the default choice of $d = 1$  in the first order method is legitimate.  
Investigating theoretical guarantees for choosing the value of $d$ in the TIREX context is beyond the scope of this paper and left for future work.

\end{remark}



\section{Experiments}
\label{sec:experiments}
 
This section focuses on the practical usefulness of TIREX for finite
sample sizes based on simulated and real data.  We first give some
details about the implementation of TIREX
(Section~\ref{sec:implementTIREX}) and discuss its computational
complexity. {
  We discuss the improvement brought by TIREX over the estimation
  method proposed by~\cite{gardes2018tail}. 
}%
Second, with synthetic datasets 
of various dimensions, we explore the estimation performance of TIREX1
and TIREX2 for various values of $k$, as measured by a distance between the estimated and true
extreme SDR spaces (Section~\ref{sec:SyntheticData}). {
  On this occasion we compare the estimation performance of TIREX with
  that of its closest alternatives, namely \cite{gardes2018tail}'s
  method, CUME and CUVE. 
  Finally in Section~\ref{sec:AD} we compare TIREX with several
  existing dimension reduction tools for predicting tail events on
  several real data sets of relatively high dimension.%
}
\subsection{TIREX implementation}\label{sec:implementTIREX}
In a preliminary step common to all our experiments, the covariates
are empirically standardized and we set
$\hat Z_i = \hat \Sigma^{-1/2} (X_i - \hat m)$ with
$\hat m = n^{-1} \sum_{i=1}^n X_i$ and
$\hat \Sigma = n^{-1} \sum_{i=1}^n (X_i - \hat m) (X_i - \hat
m)^T$. Working with empirically standardized covariates to estimate an
extreme SDR space $E_e$ is equivalent to working with raw covariates
to estimate $\tilde E_e = \Sigma^{-1/2} E_e$ up to remainder terms of
order $O_{\PP}(1/\sqrt{n})$, see Section~\ref{sec:nonstandard-estimation} in the supplementary material. By abuse of notation we use the same
symbols in the present section to denote both the empirical processes
constructed with the $\hat Z_i$'s and the $Z_i$'s.

We start off by deriving an explicit, computationally efficient formula for the matrices $\widehat M_{\text{TIREX1}}, \widehat M_{\text{TIREX2}}$. 
 Let $(\hat Z_{(1)},Y_{(1)}),(\hat Z_{(2)},Y_{(2)}),\dots,(\hat Z_{(n)},Y_{(n)})$ be such that { $Y_{(1)} \geq \cdots \geq Y_{(n)}$}. From the definition of $\hat{C}_{n}$, we have 
$\hat{C}_{n}(u)
= \frac{1}{k} \sum_{i=1}^{\lceil ku \rceil} \hat{Z}_{(i)}
$.
This implies that $\hat C_n$ is piece-wise constant, more precisely for $j\in \{1,\ldots, n\}$, whenever $u \in ( (j - 1)/k  , j /k]$, we have 
$k \hat{C}_{n}(u)=  \sum_{i=1}^{j} \hat{Z}_{(i)} : = \hat S_j$. Since
$\widehat M_{\text{TIREX1}} = \sum_{j=1}^{k} \int_{(j-1)k}
^{j/k}
\hat{C}_{n}(u) \hat{C}_{n}(u)^ {\top} d u$, it follows that
\begin{equation}\label{hat_cumex}
\widehat M_{\text{TIREX1}}  = \frac{1}{k^3}\sum_{j=1}^{k}   \hat S_j  \hat S_j ^{\top}  .
\end{equation}
Evaluating the latter display requires $O(n\log(n))$ operations for sorting the $Y$'s values; $kd$ operations to compute the $\hat S_j $, $j=1,\ldots, k$ (because $\hat S_j$ can be deduced from $\hat S_{j-1}$ with one operation); and $O(kd^2)$ operations to compute the matrix $\widehat M_{\text{TIREX1}}$. The overall cost is then of order $n\log(n) + kd^2$.
Similar arguments regarding the second order matrix  $\widehat M_{\text{TIREX2}} $ lead to  the expression 
\begin{equation}\label{hat_cuvex}
\widehat M_{\text{TIREX2}} = \frac{1}{k^3}\sum_{j=1}^{k} 
 \hat T_j  \hat T_j ^{T}  ,
\end{equation}
with $T_j =  \sum_{i=1}^{j} (\hat{Z}_{(i)} \hat{Z}_{(i)}^T -I_p )$.

The final
step 
is to perform an eigen-decomposition of the estimated matrix
$\widehat M_{\text{TIREX1}} $ (\emph{resp.}
$\widehat M_{\text{TIREX2}}$).  Given the alleged dimension $d$ of
$E_e$, the vector space generated by the $d$ eigen vectors associated
to the $d$ largest eigenvalues of the matrix (with multiplicities,
assuming uniqueness of the corresponding eigen space for simplicity)
constitutes the TIREX estimate $\widehat
E_e$. 
The non standard SDR space can be estimated by multiplying the
obtained directions by $\hat \Sigma^{-1/2} $.

{
  \subsubsection*{ Computational complexity. } 
Evaluating~\eqref{hat_cumex}  requires $O(n\log(n))$ operations for sorting the $Y$'s values; $kp$ operations to compute the $\hat S_j $, $j=1,\ldots, k$ (because $\hat S_j$ can be deduced from $\hat S_{j-1}$ with one operation); and $O(kp^2)$ operations to compute the matrix $\widehat M_{\text{TIREX1}}$. The overall cost is then of order $n\log(n) + kp^2$. Similarly the overall cost for  $\widehat M_{\text{TIREX2}}$ is of order $n\log(n) + kp^4$. Finally the eigen-decompostion based on SVD requires $O(p^3)$ operations.

In contrast the estimation procedure proposed in~\cite{gardes2018tail}
relies on an optimization strategy over a $p-d$-dimensional grid where
$d$ is the reduced dimension, and has an important computational cost
when $d>1$ according to the author (see Sections 3.2 and 4.1 of the
cited reference).  The existing implementation of
\cite{gardes2018tail}'s method is restricted to  $d=1$ and the
experiments conducted in that paper are limited to $p = 4$.  Whether
it is possible to bypass the
curse of dimensionality in~\cite{gardes2018tail}'s framework remains
an open question. For these reasons we limit our comparison with
\cite{gardes2018tail}'s method in our experiments to low dimensional
examples, Models A,C, introduced below.
\subsection{Performance for tail SDR estimation,  synthetic data}\label{sec:SyntheticData} We
  consider three 
  particular instances of the mixture model presented in
  Section~\ref{sec:MixtureGeneric}. 
  The heavy tailed noise variables $\zeta_j, j\le d$ follow identical
  Pareto distributions, $\PP[\zeta_j >t] = t^{-\alpha_2}$ with
  $\alpha_2= 10$.  The short-tailed noise variables
  $\epsilon_j, j\le p-d$ are exponentially distributed,
  $\PP[\epsilon_j>t ]= e^{-\alpha_1 t }, t>0$, with rate parameter
  $\alpha_1 = 10$. The variables
  $(\zeta_j, j\le d ; \epsilon_j, j\le p-d)$ are independent.

  \noindent {\bf Model A.} 
  We consider Case~(i) from the  generic example (continuous covariates) 
  with $\theta=0.5$, $a=1,b=10$.  For simplicity we take all covariate variables
  uniformly distributed over the interval $[a,b] = [1,10]$.  Recall
  that in this context, both \TCI and \TCIG hold. Then according to
  both definitions the $d$-dimensional subspace of $\rset^p$ generated
  by the canonical basis vectors $(e_{p-d+1}, \ldots, e_p)$ is an
  extreme SDR space. We set $p=2$, $d=1$.

\noindent{\bf Model B. } 
Here we set $p=30,d=5$, all other setup remains unchanged comparing with Model A. 

\noindent {\bf Model C. } 
We use the distribution described in Case~(ii) from
Section~\ref{sec:MixtureGeneric}, where the covariates
 are Bernoulli variables. In this context, \TCI holds but \TCIG does
not. We set the Bernoulli parameter to $\tau=0.5$.  To maintain the comparability between TIREX and \cite{gardes2018tail} we set
$p=2,d=1$. 

 \subsubsection*{ Experimental setting.} 
 The sample size is set to $n = 10^4$ for Models A and C, and to
 $n=10^5$ for Model~B. The TIREX matrices following \eqref{hat_cumex}
 and \eqref{hat_cuvex} are computed for $150$ different
 values of $k$ within the range $\llbracket n/100, n \rrbracket$.
 The orthogonal projection on the subspace generated by their first $d$  eigen vectors constitutes our estimates $\hat P_e$.  
 In other words  we consider
 for simplicity that $d$ is known by the user, as discussed in  Remark~\ref{rem:choicekd}. 
 The quality of the estimator is
 measured by the squared Frobenius norm of the
 error, 
$ \| \hat P_e- P_e\|_F^2$. 
We evaluate the squared bias
$\| P_e - \mathbb E[ \hat P_e]\|^2_F$, the variance
$\mathbb E[\|\hat P_e - \mathbb E[ \hat P_e]\|^2_F]$, and the MSE
$\mathbb E[\|\hat P_e - P_e\|^2_F]$ using TIREX, based on $N = 200$ repetitions. Thus the maximum relative error of the MSE estimate, \ie the maximum
standard deviation of the estimate divided by the estimate itself, over all
models and all values of $k$, is $0.11$, which is sufficiently small
for a qualitative interpretation of the results.

In addition we compare the relative performances of TIREX1 and 
\cite{gardes2018tail}'s method for Models A and C. 
We leave TIREX2 outside the comparison because  our results (Figure~\ref{fig:fig1}) show  that TIREX1 is a better option in this setting. 
To  alleviate the computational cost we
perform only $N=100$ repetitions and we estimate the projectors for two 
values of $k$, namely $k = n^{2/3} \approx 464$ as recommended in
\cite{gardes2018tail} and  $k= 2000$  which is close to the value minimizing the  MSE with TIREX1 for both models, considering our results below. 

{\bf Results.}  Figure~\ref{fig:fig1} displays the squared bias, variance
and MSE for TIREX1 and TIREX2 as a function of $k$. The curves  illustrate the typical bias-variance trade-off in Extreme Value Analysis regarding the choice of $k$, and confirm  the findings of  Corollary~\ref{cor:weakCV-SIREX-SAVEX}.  Small values of $k$ are associated with large variance, while large values result in a large bias.  Notice that choosing $k=n$ with TIREX1 (\emph{resp.} TIREX2) amounts to applying the standard SIR method CUME  (\emph{resp.} CUVE). Our results show  that the MSE in this case is typically much larger  (due to the bias) than with moderate  $k$'s, namely with $k \approx 2000$ for $n=10^4$ and $k\approx 15 000$ for $n= 10^5$.  

In some cases,  comparatively larger variances occur for $k \approx n/2$. 
We interpret this as an unstable transitional regime between two extremal behaviors: On the one hand, for small values of $k$, only the very largest values  of $Y$ are selected. These  are mostly generated by the second  component $Y_2$  of the mixture model, the heavy-tailed one. 
On the other hand when $k$ is large, both components $Y_1,Y_2$ are equally involved in the computation of $M_{\text{TIREX}}$. 


The variance attached to the second order method TIREX2 tends to be larger than
that of the first order method TIREX1. 
However, when the dimension of the extreme SDR space is greater than one (Model B), TIREX1 fails to recover more than one direction,  and TIREX2 is preferable. 
This fact illustrates the conclusion of Theorem~\ref{th:SIR_extreme}, see also Lemma~\ref{lem:consistentMTIREX}, 
where  a single vector (or a rank-one matrix) is identified in the
limit. TIREX2 does not suffer from this flaw since  the associated limit
in Lemma~\ref{lem:consistentMTIREX} is a matrix offering potentially
more than one direction in the SDR space.
As a conclusion, one should definitely prefer TIREX1 over TIREX2 when
the extreme values of $Y$ are known to be explained by a
single 
linear combination of $Z_1,\dots,Z_p$. Otherwise it is necessary to
resort to TIREX2 to discover additional directions, even though the
estimates may have a higher
variance.  

\begin{figure}[h]
\includegraphics[width=0.99\linewidth]{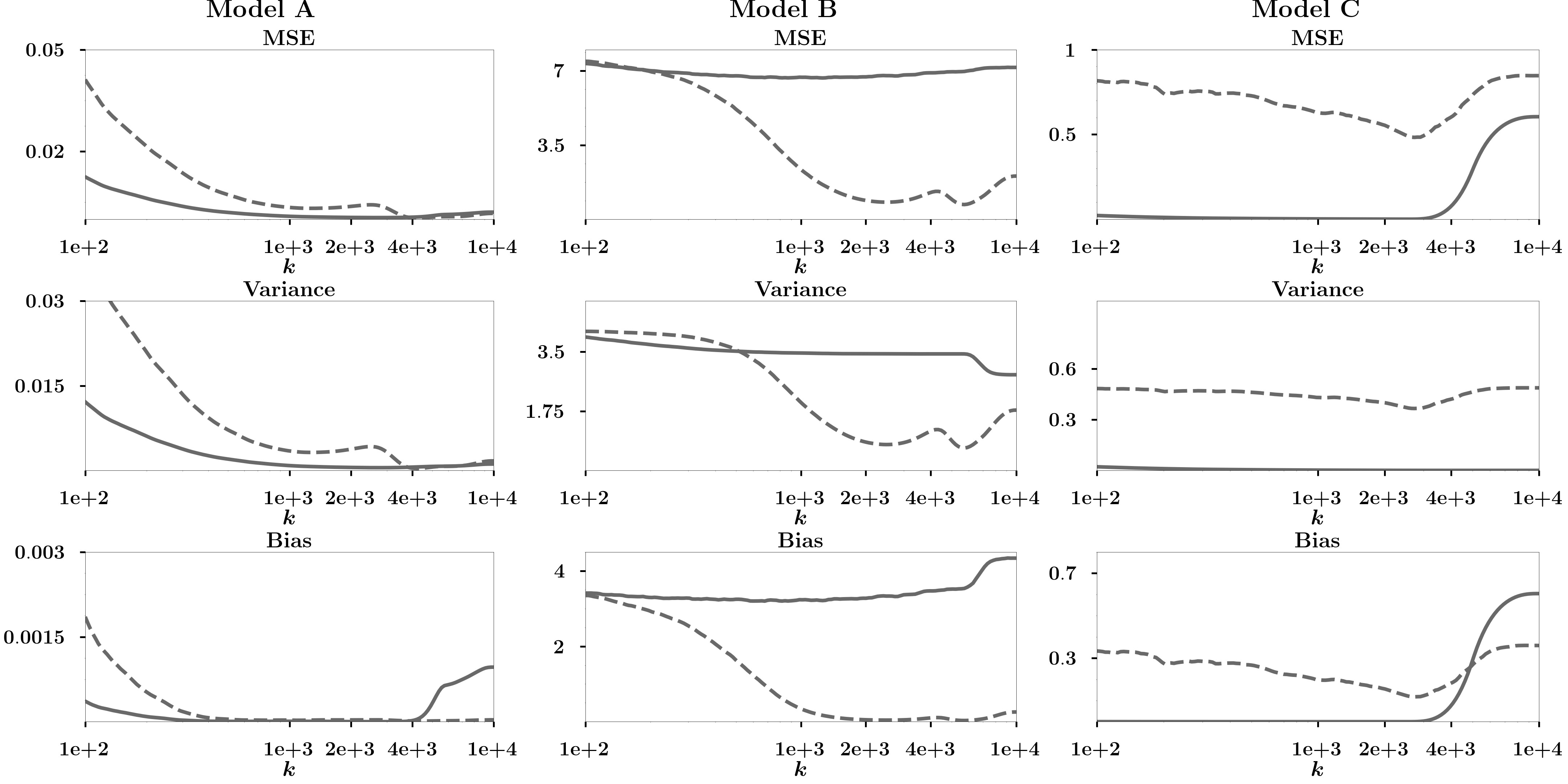}
\caption{Performance in terms of Frobenius norm of the error, as a function of $k$, with TIREX1 (solid line) and TIREX2 (dotted line), in Models A,B,C. Mean squared error, bias and variance computed over $100$ repetitions. 
}
\label{fig:fig1}
\end{figure}

Table~\ref{fig:simu-compare-gardes} displays the results of the
comparison with \cite{gardes2018tail}'s method in terms of MSE and
execution time. In Model A where \cite{gardes2018tail}'s assumptions
are satisfied, \cite{gardes2018tail}'s method  performs better than TIREX for
the two values of $k$ considered. However its execution time, even in this low dimensional setting is several
orders of magnitude higher than that of TIREX.
 In Model C, as suggested by the theory, \cite{gardes2018tail}'s method fails to recover the tail SDR space (in the sense of \TCI, not \TCIG). 
 By contrast, TIREX can recover the tail SDR space within very short execution time.

 \begin{table}[h]
   \centering
   	\begin{tabular}{|c|c|c||c|c|}
          \hline
          &{\tiny \bf Model A, TIREX1 } & {\tiny \bf   Model A, \cite{gardes2018tail} } & {\tiny \bf    Model C, TIREX1 } &{\tiny \bf   Model C, \cite{gardes2018tail} } \\ \hline
          $k=464$   &  $2. 10^{-3} \,( 2\, s )$    & $4 . 10^{-4} \,( 6\, h )$   & $4 . 10^{-3} \,( 2.3 \, s)$ &   $1  \,(4.3 \,h )$ \\ \hline
          $k=2000$  & $5.10^{-4}\, (1.7 \,s)$ & $ 5.10^{-5}\, (6.5 \,h) $   & $9. 10^{-4}\,(3.2\, s )$ &   $0.8\,( 8.5\, h)$ \\ \hline 

	\end{tabular}
	\caption{
          MSE  for  TIREX and \cite{gardes2018tail}'s
          method in  Models A and C,  $100$ replications. Execution times on a standard laptop are in brackets, with $h$ and $s$ indicating hour and second respectively.
        } 
	\label{fig:simu-compare-gardes}
\end{table}


\subsection{Predicting tail events  with TIREX on real datasets}\label{sec:AD}

We now investigate the relevance of TIREX as a dimension reduction
tool for  predicting unusually large values of $Y$.  As
  explained in Remark~\ref{rem:TCI_AMrisk}, this may be viewed as a
  classification task: predict an exceedance $\{Y>y\}$ with the help
  of $p$ covariates $X\in \mathbb R^p$.  Reducing the dimension allows
  to escape the curse of the dimensionality using the projected
  covariates, however it generally induces a bias which may influence the
  (weighted) risk of an error. The most important observation in
  Remark~\ref{rem:TCI_AMrisk} is that, if $\tailindep{Y}{X}{P_eX}$,
  the bias term vanishes in the limit $y\to\supy$. Since TIREX aims
  precisely at estimating $P_e$ such that $\tailindep{Y}{X}{P_eX}$, a
  reasonable hope is that it would generally perform better than other
  dimension reduction algorithms targeting different reduction
  subspaces $P \neq P_e$ that would not enjoy this property.
  
\subsubsection*{Experimental setting.}
We follow a two-steps procedure: first, run a dimension reduction
algorithm (TIREX or another existing method) and project the
covariates $X_i$ on the estimated SDR space; second apply a
classification algorithm to predict the event $Y_i > y$ with the help
of the projected covariates.  For all dimension reduction methods
entering the comparison, the dimension of the reduced subspace is set
to $d=2$.

Throughout our experiments the second step is fixed: We use a nearest
neighbors algorithm with a number of neighbors set to $5$.  In the end
the performance of the competing dimension reduction methods is
measured in terms of the AM risk~\eqref{eq:AMrisk} and the AUC (Area
under the ROC Curve) of the nearest neighbors classifier trained on
the reduced covariates.   The number of observations $k$ in TIREX is
selected based on $5$-fold cross-validation with the AUC criterion.

\subsubsection*{Competitors.}   TIREX is compared with several
  alternative methods using the full dataset for estimation, not only
  the subset associated with the largest values of the target.  Namely
  we consider in a supervised setting the standard SDR estimates
  obtained with the CUME and CUVE methods introduced in Section~\ref{sec:background-SIR}. In an unsupervised setting we
  consider routinely used methods available in the Python Scikit-learn
  package~\cite{scikit-learn}, namely Principal Component Analysis
  (PCA), Singular Value Decomposition (SVD) which is a non-centered
  version of PCA, Locally Linear Embedding (LLE), and Isomap (IMP).
  The latter two methods are non-linear generalizations of of PCA
  (\cite{Roweis-LLE},~\cite{Tenenbaum-IMP}, see also
  \cite{chojnacki2009note,NIPS2003-bengio}) which are widely applied
  in many contexts such as data visualization \cite{LLE-vis,Tenenbaum-IMP},
  or classification \cite{LLE-IMP-classif}, among others.  Considering the
  dimensions $p \in \llbracket 18, 103 \rrbracket $ of the datasets
  described below, \cite{gardes2018tail}'s method for dimension
  reduction could not be included in the comparison for the
  algorithmic complexity reasons described above. %
} %

\subsubsection*{Data sets.}  Eight datasets are used. Three of them come  from
the UCI repository\footnote{\url{https://archive.ics.uci.edu}}:
\textit{Residential} ($372$ apartment sale prices, with $103$
covariates); \textit{crime} ($1994$ per capita violent crimes with
$122$ socio-economic covariates); \textit{Parkinsons} ($5875$ voice
recordings along with $25$ attributes).\,Three other datasets come
from the Delve
repository\footnote{\url{http://www.cs.toronto.edu/~delve/data/datasets.html}}:
\textit{Bank} ($8192$ rejection rates of different banks, with $32$
features each);\,\textit{CompAct} ($8192$ CPU's times with $27$
covariates); \textit{PUMA32} ($8192$ angular accelerations of a robot
arm, with $32$ attributes). Finally, two other data are obtained from
the LIACC
repository\footnote{\url{https://www.dcc.fc.up.pt/~ltorgo/Regression/DataSets.html}}:
\textit{Ailerons} ($13750$ control action on the ailerons of an
aircraft with $40$ attributes) and \textit{Elevator} ($16559$ control
action on the elevators of an aircraft with $18$ attributes). 

\subsubsection*{Results.}
For all
datasets, $y$ is chosen equal to
the 
$0.98$-quantile of the target $(Y_i)_{i=1,\ldots, n}$ except for
\textit{Residential} where the $0.90$-quantile has been used to
counterbalance the small sample
size. 
The results in terms of AM risk and AUC are  summarized in Tables~\ref{fig:Real-data-AMRISK} and \ref{fig:Real-data-AUC} respectively. In the vast majority of cases,  TIREX1 or TIREX2  performs better than the other methods. On these examples, TIREX1 is often superior to TIREX2, which indicates that the added flexibility introduced by the second order moments does not compensate for the increased variance.  

\begin{table}[h]
  \centering
  \begin{tabular}{|l|l|l|l|l|l|l|l|l|}
    \hline
    & TIREX1 & TIREX2 & CUME  & CUVE & PCA & SVD & LLE & IMP \\ \hline
    Bank & 0.434 & \textbf{0.378} & 0.42 & 0.392 & 0.418 & 0.474 & 0.486 & 0.432 \\ \hline
    Crime & \textbf{0.412} & 0.5 & 0.471 & 0.47 & 0.502 & 0.469 & 0.47 & 0.5 \\ \hline
    CompAct & \textbf{0.208} & 0.279 & 0.287 & 0.313 & 0.242 & 0.243 & 0.271 & 0.253 \\ \hline
    Residential & \textbf{0.158} & 0.353 & 0.421 & 0.447 & 0.479 & 0.479 & 0.49 & 0.49 \\ \hline
    Parkinsons & \textbf{0.252} & 0.346 & 0.268 & 0.346 & 0.469 & 0.469 & 0.455 & 0.47 \\ \hline
    Puma32 & 0.492 & 0.501 & 0.5 & 0.5 & 0.5 & 0.5 & 0.501 & \textbf{0.49} \\ \hline
    Elevators & \textbf{0.446} & \textbf{0.446} & 0.471 & 0.463 & 0.5 & 0.5 & 0.5 & 0.5 \\ \hline
    Ailerons & \textbf{0.307} & 0.329 & 0.314 & 0.33 & 0.498 & 0.499 & 0.498 & 0.501 \\ \hline
	\end{tabular}
	\caption{
          AM risk  of the nearest neighbors classifier with reduced covariates obtained with different dimension reduction methods. } 
	\label{fig:Real-data-AMRISK}
\end{table}

\begin{table}[h]
  \centering
  \begin{tabular}{|l|l|l|l|l|l|l|l|l|}
    \hline
    & TIREX1 & TIREX2 & CUME  & CUVE & PCA & SVD & LLE & IMP \\ \hline
    Bank & \textbf{0.771} & 0.696 & 0.698 & 0.684 & 0.736 & 0.689 & 0.608 & 0.65 \\ \hline
    Crime & 0.666 & 0.67 & 0.616 & 0.686 & 0.678 & \textbf{0.773} & 0.672 & 0.661 \\ \hline
    CompAct & 0.893 & 0.887 & \textbf{0.899} & 0.871 & 0.876 & 0.874 & 0.868 & 0.885 \\ \hline
    Residential & \textbf{0.902} & 0.827 & 0.674 & 0.745 & 0.667 & 0.659 & 0.666 & 0.694 \\ \hline
    Parkinsons & \textbf{0.901} & 0.818 & 0.852 & 0.82 & 0.742 & 0.753 & 0.743 & 0.748 \\ \hline
    Puma32 & \textbf{0.711} & 0.578 & 0.616 & 0.515 & 0.587 & 0.577 & 0.537 & 0.547 \\ \hline
    Elevators & 0.686 & \textbf{0.694} & 0.615 & 0.672 & 0.528 & 0.537 & 0.514 & 0.514 \\ \hline
    Ailerons & \textbf{0.853} & 0.834 & 0.828 & 0.832 & 0.502 & 0.515 & 0.514 & 0.525 \\ \hline
	\end{tabular}
	\caption{ AUC of the nearest neighbors classifier with reduced covariates obtained with different dimension reduction methods.} 
	\label{fig:Real-data-AUC}
\end{table}



\begin{supplement}
  \stitle{}
  \sdescription{The supplementary material  placed below the bibliography 
    contains proofs, additional examples and discussions regarding existing notions of Tail Conditional Independence, and extensions to non-standardized covariates. }
\end{supplement}




\bibliographystyle{imsart-nameyear} 
\bibliography{bib_SIR}      


\pagebreak

\appendix
\counterwithin{lemma}{section}
\counterwithin{proposition}{section}
\counterwithin{remark}{section}
\counterwithin{cor}{section}

\begin{center}
\noindent  {\LARGE \centering \bf Tail Inverse Regression:  dimension reduction for prediction of extremes.  \\ Supplementary material.}
\end{center}
~\\

          \noindent This supplement 
contains proofs, additional examples and discussions regarding existing notions of Tail Conditional Independence, and extensions to non-standardized covariates.
        Section and equation numbers in the supplement start with a letter, to distinguish them from those in the paper.





\section{Proofs for Remark~\ref{rem:TCI_AMrisk}}\label{sec:proof_rem1}
In this section, for the sake of completeness,  we prove  two facts regarding classification with the  AM risk in the \emph{full problem} defined in Remark~\ref{rem:TCI_AMrisk} from the main paper. First the classifier 
\begin{equation}
  \label{eq:ExpressionBayesClassif}
  h^*(x) =  \un\{\eta(x) > \pi \}
\end{equation}
is a minimizer of the AM risk ; Second,  the associated Bayes risk is given by
\begin{equation}
  \label{eq:BayesRiskProof}
  \AMR(h^*) = \EE[ \min\Big( \frac{\eta(X)}{\pi}, \frac{1 - \eta(X)}{1-\pi} \Big) ]. 
\end{equation}

We introduce the AM loss function
\[
\AMl(\hat t, t) = \frac{1}{1-\pi} \un\{ \hat t = 1, t=0\} + \frac{1}{\pi} \un\{ \hat t = 0, t = 1 \}
\]
so that for any classifier, $\AMR(h) = \EE[ \AMl(h(X), T) ] $. Consider now  the conditional AM risk
\[\widetilde{\AMR}(h,x) = \EE[ \AMl(h(X), T) \given X=x ],  \]
thus $\AMR(h) = \EE[\widetilde{\AMR}(h,X)]$. We also have 
\begin{align}
  \widetilde{\AMR}(h,x) & = \frac{1}{1-\pi}\un\{h(x) = 1\}(1 - \eta(x)) +
                          \frac{1}{\pi}\un\{h(x) = 0\} \eta(x) \nonumber \\
  & = \frac{1 - \eta(x)}{ 1-\pi} + \un\{h(x) = 0\} \Big[  \frac{\eta(x)}{\pi} - \frac{1 - \eta(x)}{1-\pi} \Big].  \label{eq:conditionalAMRisk}
\end{align}
Also, the classifier in \eqref{eq:ExpressionBayesClassif} may be written equivalently as
$  h^*(x) =  \un\Big\{\frac{ \eta(x)}{\pi} > \frac{1 - \eta(x)}{1-\pi} \Big\}$. 
Thus for any classifier $h$, we may write the difference in conditional risks as
\begin{align*}
  \widetilde{\AMR}(h,x) - \widetilde{\AMR}(h^*,x)
  & = \frac{\eta - \pi}{\pi(1-\pi)} \big[  \un\{h(x) = 0 \} - \un\{h^*(x) = 0 \}\big] \\
  & = \left| \frac{\eta - \pi}{\pi(1-\pi)} \right| \un\{ h(x) \neq h^*(x) \}
\end{align*}
The latter display is nonnegative, which shows that $h^*$ defined in~\eqref{eq:ExpressionBayesClassif} indeed minimizes the AM risk.
Turning to our second claim, notice that we may write, using~\eqref{eq:conditionalAMRisk}, 
\begin{align*}
\widetilde{\AMR}(h^*,x)&  =
\begin{cases}
  \eta(x)/\pi & \text{ if } \eta(x)/\pi  > (1-\eta(x))/(1-\pi)\\
  (1 - \eta(x))/(1-\pi)& \text{ otherwise } 
\end{cases} \\
                         & = \min\Big( \frac{\eta(x)}{\pi}, \frac{1-\eta(X)}{1-\pi} \Big). 
\end{align*}
This proves~\eqref{eq:BayesRiskProof}.


    \section{Proofs for Section~\ref{sec:MixtureGeneric} and additional comments }\label{sec:compareGardes_long}
    In this section we provide the full proofs regarding our examples and counter-examples from Section~\ref{sec:MixtureGeneric} regarding the generic mixture model. On this occasion we  
    conduct a thorough comparison between the two definitions of tail conditional independence \TCI and \TCIG, see Equations~\ref{eq:newdefTCI} and~\ref{eq:defGardes} in the main paper. 
  For convenience write
  $S(y)= \PP(Y>y); S(y, W)= \PP(Y>y|W); S(y,W,V) =
  \PP(Y>y|W,V)$.  
  The relevant quantities  are respectively the ratios
  \begin{equation}\label{eq:ratios-TCI-TCIG}
R(y,V,W) =  \frac{S(y, V,W)- S(y, W)}{S(y)} \;, \text{ and } \tilde R(y,V,W) =   \frac{S(y, V,W)- S(y , W)}{S(y, W)}. 
\end{equation}
The \TCI condition is that    $\EE | R(y,V,W) | \to 0$ as $y\to\supy$, 
whereas \TCIG means that  $\tilde R(y,V,W) \to 0$ as $y\to\supy$,  almost surely.
 Notice already  that our criterion
\eqref{eq:newdefTCI} is an integrated version of \eqref{eq:defGardes},
with a weight function
\begin{equation}
  \label{eq:defineRho_tcig}
\rho(y,W) = S(y,W)/S(y),   
\end{equation}
such that 
$\rho(y,W)\ge 0$ and $\EE[\rho(y,W)] = 1$ for all $y$. Namely, \TCI means that 
\begin{equation}
  \label{eq:relationTCI-TCIG}
  \EE \, \left| \tilde R(y,V,W) \rho(y,W) \right|  \xrightarrow[y\to\supy]{} 0
\end{equation}

   
\subsection{Additional notations  regarding the generic mixture model from Section~\ref{sec:MixtureGeneric}}\label{sec:additionalNotationMixture}
  We introduce in the context of Section~\ref{sec:MixtureGeneric} from the main paper  the additional notations
  \begin{align*}
    S_1(y) = \PP(Y_1>y) = \int S_1(y,v) \ud P_V(v) \;,\quad     S_2(y) =  \PP[Y_2>y] = \int S_2(y,w) \ud P_W(w). 
  \end{align*}
  With these notations, using the independence assumption regarding the pair $(V,W)$ we may write
  \begin{gather*}
    S(y,v,w) = \theta S_1(y,v) + (1-\theta)S_2(y,w) \;;\quad 
    S(y,w)  = \theta S_1(y) + (1-\theta)S_2(y,w) \;;\\
    S(y)  = \theta S_1(y) + (1-\theta)S_2(y). %
  \end{gather*}
Thus, 
  the ratios $R$, $\tilde R$ defined at the beginning of this section and involved in \TCI and \TCIG write respectively
  \begin{align}
    R(y,v,w) =  \frac{\theta(S_1(y,v) -S_1(y) )}{
    \theta S_1(y) + (1-\theta)S_2(y)}\,, \qquad 
    \tilde R(y,v,w)  =  \frac{\theta (S_1(y,v) -S_1(y) )
                      }{\theta S_1(y) + (1-\theta)S_2(y,w)}\,.  \label{eq:ratioWithS_12}
  \end{align} 
  Notice already that
  \begin{align}
    | R(y,v,w) |& 
    \le \frac{\theta}{1-\theta}\frac{S_1(y,v) + S_1(y)}{S_2(y)} \,,\label{eq:majorR}\\
    | \tilde R(y,v,w)| & \le \frac{\theta}{1-\theta}\left( \frac{S_1(y,v)}{S_2(y,w)} + \int \frac{S_1(y,v')}{S_2(y,w)} \ud P_V(v')\right)\,.\label{eq:majorTildeR}  
  \end{align}

  Finally, specializing to the case where $Y_1$ and $Y_2$ follow the mixture model described in the same section of the main paper, the conditional survival functions for  $Y_1,Y_2$ are,  for $y>b$, 
  \begin{equation}
    \label{eq:S_12_specific}
    \begin{aligned}
      S_1(y,v) &= \sum_{i=1}^{p-d}  \un\{v_i>0\} \pi_{i}^{1}  S_{\epsilon}(y/v_i) 
                 \;,\qquad
      S_2(y,w) = \sum_{j=1}^{d}  \un\{w_j>0\} \pi_{j}^{2} S_{\zeta}(y/w_j)  
    \end{aligned}
  \end{equation}


  
  We now discuss the main differences between the two
  definitions. Natural questions to ask are \emph{(i)} whether one
  definition is more appropriate than the other depending on the
  context ; \emph{(ii)} whether one condition is stronger than the
  other, possibly under additional assumptions.

As for Question \emph{(i)}, in spirit,   as reflected by the equivalent condition~(\ref{eq:relationTCI-TCIG}), \TCI is
comparatively more sensitive to values $ W = w $ such that the conditional probability of an exceedance
 $Y>y$  is large, which is an appealing feature for identifying tail risk factors as described in the
introduction.   On the other hand, one advantage of \TCIG's scaling is that the 
ratio $\tilde R$ introduced at the beginning of this section is a \emph{relative} deviation, which is arguably easily
interpretable. However \TCIG's criterion 
takes into account  \emph{all} possible
values $w$, even those such that the conditional distribution of $Y$
given $W=w$ is shorter tailed than the marginal distribution of
$Y$. 
The focus in \TCIG is not exactly on the tail of $Y$'s distribution,
but rather on the tails of the conditional distributions of $Y$ given
$W$.

Before turning to Question \emph{(ii)}, we discuss the differences between the two conditions in terms of mode of convergence.  
\subsection{Convergence almost-surely or in expectation in \TCIG or \TCI}\label{sec:CV-as-expect-in-TCI}  
Almost sure convergence $\tilde R(y,V,W) \to 0 $ as $y\to\supy$
implies $\EE|\tilde R(y,V,W)|\to 0$. Indeed by conditioning on $W$, we
have $\EE\,[\tilde R(y,V,W)] = 0$ so that, denoting by $z_+$
(resp. $z_-$) the negative (resp. positive) part of a real $z$, it
holds that $\EE\, [\tilde R(y,V,W)_+] =\EE\, [\tilde R(y,V,W)_-] $. As
a consequence $$\EE\,|\tilde R(y,V,W)| = 2\EE\,[\tilde R(y,V,W)_-].$$
However for all $y,v,w$, $\tilde R(y,v,w) \ge -1$ so that
$0\le \tilde R(y,V,W)_- \le 1$. By dominated convergence, if
\eqref{eq:defGardes} holds, then also $\EE\,[\tilde R(y,V,W)_-]\to 0$
and the above display implies that $\EE\,|\tilde R(y,V,W)| \to 0$ as
well.  This argument is not valid regarding the tail behaviour of
$R(y,V,W)$ because it is not true that $R(y,V,W)\ge -1$ almost surely.

We are now ready to examine Question \emph{(ii)}, that is, whether one condition (\TCI or \TCIG) implies the other, in general or under simplifying assumptions. 
\subsection{Special case: discrete covariates with finite support }\label{sec:discreteCovar}
In order to build up the intuition,  consider the special case where the
covariates have a finite support. This is a sensible assumption for
real life applications where observations are discretized.

We thus consider here  finitely supported covariates $V \in \{v_1,\ldots, v_m\}$,
$W \in \{w_1,\ldots, w_n\}$. Denote
$ p(v_i) = \PP(V = v_i), p(w_j) = \PP(W = w_j), p(v_i, w_j) = \PP( V =
v_i, W = w_j)$. Assume for simplicity that for all
$(i,j)\in\{1,\ldots,m\}\times \{ 1,\ldots, n\}$, we have
$p(v_i, w_j)>0$.

First, in this case, almost sure convergence and convergence in
expectation are equivalent for both ratios $R$ and $\tilde R$
introduced at the beginning of this section. In other words
  \begin{align}
    \EE|R(y,V,W)|\xrightarrow[y\to\supy]{}0
    &\iff |R(y,V,W)|\xrightarrow[y\to\supy]{}0, \text{ almost surely} \;; 
      \label{eq:equiv_R} \\
    \EE|\tilde R(y,V,W)|\xrightarrow[y\to\supy]{}0
    &\iff |\tilde R(y,V,W)|\xrightarrow[y\to\supy]{}0, \text{ almost surely} \;.
      \label{eq:equiv_tildeR}
  \end{align}
  Indeed
  \[
    \EE|R(y,V,W)| = \sum_{i=1}^m \sum_{j=1}^n p(v_i,w_j)\Big|
    \frac{S(y,v_i,w_j) - S(y,w_j) }{S(y)} \Big|.
\]
The latter display converges to $0$ as $y\to\supy$ if and only if each
terms in the finite summation does, that is, if and only if
$\forall(i,j), R(y,v_i,w_j)\to 0$ as $y\to\supy$. This proves
\eqref{eq:equiv_R}, and the argument for \eqref{eq:equiv_tildeR} is
similar.

Second, \TCIG implies \TCI, meaning that our definition is weaker than
\cite{gardes2018tail}'s in this discrete setting. 
To see this, in view of the equivalence between $L^1$ and almost sure
convergences, it is enough to show that the ratio
$R(y,v_i,w_j)/\tilde R(y,v_i,w_j)$ is uniformly upper bounded when
$y, i$ and $j$ vary.  However for all $(y,i,j)$,
\begin{align*}
  \frac{R(y,v_i,w_j)}{\tilde R(y,v_i,w_j)}
  & = \rho(y,w_j)=  S(y,w_j) / S(y)  
    = \frac{S(y,w_j)}{\sum_{k=1}^n p(w_k) S(y, w_k)}  \le \frac{1}{ p(w_j)}
    \le 1/ \min_{k\le n} p(w_k) < \infty. 
\end{align*}
As a consequence, if $\tilde R(y, V,W)\to 0$ almost surely, then also
$R(y,V,W) \to 0$ almost surely as $y\to\supy$ and the result follows.

  \subsection{Example in the mixture model where both \TCI and \TCIG hold}\label{sec:Example_bothTCI-TCIG}
  We consider the setting of Section~\ref{sec:MixtureGeneric} from the main paper, and in particular  the case where the lower bound of the support of each $W_j$ is positive, $a>0$.
  
We verify that the upper bounds \eqref{eq:majorR} and \eqref{eq:majorTildeR} uniformly converge to $0$. 
First, using \eqref{eq:S_12_specific}, we have 
\begin{align*}
  \frac{S_1(y) + S_1(y,v)}{S_2(y)} &\le 2 \frac{\sup_{v\in[a,b]^{p-d}} S_1(y,v)}{\inf_{w\in[a,b]^d} S_2(y,w)} 
                         \le 2 \frac{S_\epsilon(y/b)}{S_\zeta(y/a)}, 
\end{align*}
where the right-hand-side converges to $0$ as $y\to\infty$ under Condition~\eqref{eq:differentTails}. 
Thus the upper bound in~\eqref{eq:majorR} uniformly converges to $0$ and \TCI\ holds by dominated convergence.

Turning to $\tilde R$, we also have 
\begin{align*}
  \sup_{(v,w)\in[a,b]^{p}} \frac{S_1(y,v)}{S_2(y,w)} \le 
  \frac{\sup_{v\in[a,b]^{p-d}} S_1(y,v)}{\inf_{w\in[a,b]^d} S_2(y,w)}
  \le  \frac{S_\epsilon(y/b)}{S_\zeta(y/a)}   \xrightarrow[y\to\infty]{} 0. 
\end{align*}
Thus, by dominated convergence the right-hand-side of \eqref{eq:majorTildeR} converges to $0$ as $y\to\infty$ so that \TCIG holds as well. 

In the general case the
situation is much more complex and it turns out that neither condition
implies the other, as revealed by the counter-examples constructed in the next two subsections. 

\subsection{Counter-example in the mixture model where \TCI holds but \TCIG does not}\label{sec:TirexNotImpliesGardes}
In contrast to the latter subsection, we now consider the case where
the support of the $W_j$'s includes $0$, so that $a=0$. Namely we take
each variable $V_j,W_j$ following a binary Bernoulli distribution with
parameter $\tau \in (0,1)$
Thus
$\PP[W = (0,\ldots,0) ]= (1-\tau)^d>0$. Notice already that the
right-hand side of \eqref{eq:majorTildeR} is not bounded because  $S_2(y,w= (0,\ldots,0) ) = 0$ for $y>0$.
Also, from~\eqref{eq:ratioWithS_12}, 
\begin{align*}
  \tilde R(y,v,w=(0,\ldots, 0)) &= \frac{ S_1(y,v) - S_1(y) }{S_1(y)}.
\end{align*}
In this specific example  $Y_1$ and $Y_2$ have point masses at $0$ and we have for $y>0$,  
$S_1(y) = \sum_j \pi_j^{1}\tau S_\epsilon(y) = \tau S_\epsilon(y)$ 
while for $v_1= (1,\ldots, 1)$, $S_1(y,v_1) =  \sum_j \pi_j^{1} S_\epsilon(y) = S_\epsilon(y)$.
Thus in the above display,  $\tilde R(y,v_1, 0) = (1 - \tau)/\tau$ for all $y>1$ and \TCIG does not hold.

Finally we  show that \TCI holds by examining the right-hand side
of~\eqref{eq:majorR}. The argument above shows that
\[
\frac{S_1(y,v) + S_1(y)}{S_2(y)} \le \frac{(1+\tau)S_\epsilon(y)}{\tau S_\zeta(y)}\,.
  \]
This proves uniform convergence to~$0$
in~\eqref{eq:majorR} under Condition~\eqref{eq:differentTails}
and concludes the argument.

\subsection{Counter-example where  \TCIG holds but \TCI does not} 
\label{sec:GardesNotImpliesTirex}

In this example we depart from the mixture model forming the basis of  the two latter examples.
The idea behind 
is to build the survival functions in such a way that $\limsup_{y\to\infty} \rho(y,W) = \infty$ (see~(\ref{eq:defineRho_tcig}) for the definition of $\rho$), with probability one, while \TCIG holds. 


In addition to  the notations introduced at the beginning of this section, we introduce
the ratio
$$q(y,v,w) = S(y,v,w)/S(y,w).$$ Thus
$\tilde R(y,v,w) = q(y,v,w)-1 $ and  $R(y,v,w) = (q(y,v,w)-1)\rho(y,w)$. 
We denote respectively by $P_W, P_{V,W}$ the marginal distribution of $W$ and the joint distribution of $(V,W)$. Here we define  $V$, $W$ as independent uniform variables,   $P_W = P_V = \mathcal{U}_{[-1/2,1/2]}$ and  $P_{V,W} = P_V\otimes P_W$. We shall build $(S,\rho,q)$ such that  
h$|q(y,V,W)-1|\to 0$ as $y\to \infty$, almost surely, (so that \TCIG  holds)  while $\limsup \EE[|q(y,V,W)-1|\rho(y,W)]>0$ as $y\to\infty$ (so that \TCI  does not hold).
  
The functions $S(y), q(y,v,w), \rho(y,w)$ define  a  joint distribution of $(Y,V,W)$ with no mass at the right end point of $Y$ 
if conditions~\eqref{cond_S}~\eqref{cond_rho} and~\eqref{cond_q} below hold.
\begin{gather}
  S \text{ is non-increasing }, \quad \lim_{y\to\supy} S(y) = 0,\quad  S(y)\ge 0 ;  \label{cond_S}
\end{gather}
\begin{gather}
 P_W \text{-almost surely}, \text{  the function } y\mapsto \rho(y,W) S(y)   \text{ is non-increasing, and  } \nonumber \\
  \lim_{y\to\supy} \rho(y,W) S(y) = 0, \quad \rho(y,W)\ge 0, \quad \EE[\rho(y,W)] =1 \, ,\forall y ;  \label{cond_rho}
\end{gather}
\begin{gather}
  P_{V,W}\text{-almost surely}, \text{ the function } y\mapsto q(y,V,W) \rho(y,W) S(y)   \text{ is non-increasing, and  } \nonumber \\
  \lim_{y\to\supy}q(y,V,W) \rho(y,W)  S(y) = 0, \quad q(y,V,W) \ge 0, \quad  \EE[q(y,V,W)\given W] =1 \,, \forall y. \label{cond_q} 
\end{gather}

\subsubsection{Construction of $S(y),\rho(y,w)$}\label{sec:construct_S_rho}

We let $S(y) = e^{-y}$,  $y\ge 0$,  and 
we construct $\rho$ such that   $\PP[\limsup_{y\to \infty} \rho(y,w) = \infty] = 1$ while \eqref{cond_rho} is satisfied.  To this end  define for $n\ge 2$, and $0\le j\le n$, 
\begin{equation}
  \label{eq:def_Ln}
  L_n = \sum_{k<n, k\ge 2} k^2 \;;\qquad  L_{n,j} = L_n + jn. 
\end{equation}
Thus $L_2=0, L_3 = 4$, $L_n \le n^3$ for $n\ge 2$ and $ L_{n,n} = L_{n+1}$. Also $\rset_+ =
\sqcup_{n\ge 2} \sqcup_{0\le j<n} [L_{n,j}, L_{n,j+1})$. 
Also for $y\ge 0$, we denote by $(n(y), j(y))$ the unique pair of integers such that $y\in [L_{n,j}, L_{n,j+1})$.   

For $n\ge 2,0\le j<n$,  we define $\rho(y,w)$ for $y \in  [L_{n,j}, L_{n,j+1})$ and $w \in [-1/2,1/2]$ as follows: let $I_{n,j} = [1/2 - (j+1)/n, 1/2 - j/n]$, then  
\begin{equation}
  \label{eq:def_rho}
  \rho(y, w) =
  1 + w + \frac{n }{4 \pi} \sin\Big( \pi (y - L_{n,j}) / n \Big)
  \left[ \un\{w\in I_{n,j}\} - \un\{w\notin I_{n,j}\}/(n-1)\right]\,.
\end{equation}
Notice that for all $w\in[-1/2,1/2]$, the function $y\mapsto \rho(y,w)$ is continuous.
Also for all $w\in[-1/2,1/2]$ we have $\limsup_y \rho(y,w) = +\infty$. Indeed for any fixed  $n\ge 2$, let $j$ such that $w \in  I_{n,j}$. Then letting 
\[
 y_n = L_{n,j} + n/2, 
  \]
  we have $\rho(y_n, w) = w + n/(4\pi) \ge n/(4\pi) -1/2.$  The  sequence $y_n$ converges to $\infty$ and is  such that $\rho(y_n,w)\to \infty$ as $n\to \infty$,  which proves the claim.

  We now verify that the conditions gathered in \eqref{cond_rho} hold.
  \begin{enumerate}
  \item 
  First for all $y\ge 0$,
\begin{align*}
  \EE[\rho(y,W)]
  & =  1 + \EE[W ] +
   \frac{n }{4 \pi} \sin\Big( \pi (y - L_{n,j}) / n \Big)
  \left[ 1/n - (n-1)/(n(n-1))\right] =1.
                \end{align*}

\item We show that for all $y$, $\rho(y,W) \ge 1/3$ almost surely. By construction, $\rho(y,W) \ge 1/2  - \frac{n(y) }{4(n(y)-1) \pi} $. Since
  $m/(m-1)\le 2$  for $m\ge 2$, we obtain
  \[\rho(y,W) \ge 1/2  - \frac{2 }{4 \pi} \ge 1/2-1/6 = 1/3 . \]
\item 
  We now show that $y\mapsto S(y)\rho(y,w)$ is non increasing for all $w\in[-1/2,1/2]$.  Since both  $S$ and $\rho$ are continuous functions of $y$, with derivatives from the right which we denote respectively $S'(y)$ and $\rho'(y,w)$, we need to show that $\rho'(y,w) < - \rho(y,w)S'(y)/S(y)$. Here $S'(y)/S(y) = - 1$, and from the above point we obtain $- \rho(y,w)S'(y)/S(y)\ge 1/3$. 
To conclude we show that  \[\forall y\ge 0, w\in [-1/2,1/2],  \rho'(y,w) \le 1/4. \]  Let $y>0$ and $(n,j) = (n(y), j(y))$ as above.  On the one hand if
$w \in I_{n(y), j(y)}$ 
we have
$0\le \rho'(y,W) \le 1/4$. On the other hand if
$w \notin I_{n(y),j(y)}$, we have $\rho'(y,w)<0$. In both cases $\rho'(y,w)\le 1/4\le 1/3 \le -\rho(y,w) S'(y)/S(y)$, which conludes the argument.

\item Finally we verify that $\lim_{y\to\infty} \rho(y,W)S(y)=0$,
  almost surely. To see this, notice that for all
  $y>0, w\in[-1/2,1/2], |\rho(y,w)|\le 3/2 + \frac{n(y)}{4\pi}$. Now since $L_n \ge n^2$,
  $\{n: L_n\le y \} \subset \{n: n^2\le y \}$, so that 
    $  n(y)  =  \sup\{n: L_n\le y \}  \le \sup\{n: n^2\le y \}  \le \sqrt{y}. $
  Thus $|\rho(y,w)|e^{-y} \le (3/2 + \sqrt{y}/(4\pi) )e^{-y} \to 0$ as $y\to \infty$. 
  \end{enumerate}

  \subsubsection{Construction of $q(y,v,w)$}\label{sec:construct_q}
    Recall $n(y)$ from the beginning of the above paragraph. Define
    \begin{equation}
      \begin{aligned}
        q(y,v,w) =& 1 + v \bigg[\un\Big\{w>\frac{1}{2} - \frac{1}{n(y)}\Big\}+
                    \un\Big\{w\le \frac{1}{2} - \frac{1}{n(y)} \Big\}
                    \exp\Big(- \frac{y +  \left\lceil\frac{ 1}{1/2 - w }\right\rceil}{4} \Big)\bigg]
      \end{aligned}
    \label{eq:def_q}
  \end{equation}

  
We now verify that  the function $y\mapsto q(y,v,w) S(y,w) = S(y,v,w)$ is non increasing. Notice already that all the other constraints gathered in~\eqref{cond_q} are satisfied. Since for fixed $(v,w)$, both  $y\mapsto q(y,v,w)$ and  $y\mapsto S(y,w)$ are continuous,  it is enough to verify that the derivative from the right of $y\mapsto q(y,v,w) S(y,w)$ is negative or null, that is (since $q\ge 1/2$ is positive), we need to ensure that
  \begin{equation}\frac{q'(y,v,w)}{q(y,v,w)} \le - \frac{S'(y,w)}{S(y,w)}. \label{eq:condmonot}
  \end{equation}
  With our definition of $S(y,w)$ from Subsection~\ref{sec:construct_S_rho},
    \begin{align*}
    - S'(y,w)/S(y,w) &= (\rho(y,w) - \rho'(y,w))/\rho(y,w) 
                     = 1 - \rho'(y,w)/\rho(y,w)  
    \ge 1 - \frac{1/4}{1/3} 
    = 1/4. \end{align*}
  If we denote $y(w) =   y -  \left\lceil\frac{ 1}{1/2 - w }\right\rceil$, we have 
  \[
q'(y,v,w)/q(y,v,w) = -\frac{1}{4} \un\Big\{w\le \frac{1}{2} - \frac{1}{n(y)}\Big\} v \exp(- y(w)/4) / (1 + v \exp(-y(w)/4) ).    
    \]
    The above display is always less than $1/4$ so that~\eqref{eq:condmonot} holds for all $y>0$ and    $v,w \in [-1/2,1/2]$ and \eqref{cond_q} is satisfied.  This fact combined with the argument in Subsection~\ref{sec:construct_S_rho} implies that the functions $(S, \rho,q)$ define a proper joint distribution for $(Y,V,W)$.


    \subsubsection{Conclusion}
\label{sec:conclu_gardes_notimplies_tirex}
We have constructed a joint distribution for $(Y,V,W)$ in Sections~\ref{sec:construct_S_rho},~\ref{sec:construct_q},  
such that $P_{V,W}$-almost surely, $q(y,V,W)\to 1$ as $y\to\infty$, as can be seen immediately from the definition of $q$ in~\eqref{eq:def_q}. Thus $(Y,V,W)$ satisfy  \TCIG. 
However, for all $n\ge 0$, let  $y_n= L_n + n/2$ (see Subsection~\ref{sec:construct_S_rho}), so that  by construction $n(y_n) = n$. Notice that  $I_{n,0} = [1/2-1/n, 1/2]$ and for $w\in I_{n,0}$, we have $\rho(y_n,w) = 1+ w + n/(4\pi)\ge n/16$ and $q(y_n,v,w) = 1 + v$. Thus  
\begin{align*}
 \EE[|R(y_n,V,W)|] &=  \EE[ | q(y_n,V,W) -1| \rho(y_n,W)]\\
  & \ge \EE[ | q(y_n,V,W) -1| \rho(y_n,W) \un\{W >1/2-1/n\} |]  \\
  & =  \PP[W>1/2-1/n]  \EE[| q(y_n,V,W) -1| \rho(y_n,W) \given W \in I_{n,0}] \\
  & \ge  \frac{1}{n}  \EE[ |V|  n /16] \ge \EE[|V|] /16 = 1/64. 
\end{align*}
We have shown that $
\limsup_{y\to\infty} \EE[|R(y_,V,W)|] >0, $ 
    so that \TCI does not hold, which concludes the counter-example. 
    \subsection{Additive Mixture Model (Remark 2 in the main paper)}
  We end this section devoted to examples with a  full derivation of the  additive mixture example mentioned in Remark 2 from the main paper. 
    We consider here an additive mixture $Y = Y_1 + Y_2$. The first (light-tailed) component is  $Y_1 = V \in [a,b]$ ($-\infty <a<b<\infty$) ; and  the second (heavy-tailed) one is $Y_2 = W \xi$ where $ W\in [c,d]$ with $0<c<d<\infty$, and  $\xi$ has a continuous survival function   $S_\xi(y):=1-F_{\xi}(y)$ satisfying $q(y)=y^\alpha S_\xi(y)\to C $ as $y\to\infty$, for some $\alpha, C>0$. 
    In addition, we assume that $V$ and $W$ are independent.

We show that \TCI holds, that is $\tailindep{Y}{V }{W}$.    
Introducing the function
\[g(v , w ,y) = w^{-\alpha}y^{\alpha} 
  S_\xi[(y - v) / w] ,\]
 we have that
\begin{equation}
  \label{eq:unifCVex2}
  \begin{aligned}
    &\sup_{[v,w] \in [a,b]\times [c,d] } \big|g(v,w,y)-C| \\
    &= \sup_{v,w\in [a,b]\times [c,d]}  \Big| \Big(1 - \frac{v}{y}\Big)^{-\alpha}
    \Big(\frac{y - v}{w} \Big)^\alpha S_\xi\Big[\frac{y - v}{ w} \Big] - c \Big| \xrightarrow[y\to\infty]{} 0. \quad 
  \end{aligned}
\end{equation}
The last limit relation follows from  $\frac{y - v}{ w}  \to \infty$ and
$1 - v/y\to 1$, uniformly for $v,w\in [a,b]\times [c,d]$ as $y\to\infty$. We have that
\begin{align*}
&\mathbb P(Y>y | V,W) = S_\xi( (y - V) / W)  = W^\alpha y^{-\alpha} g(V,W,y) \\
  & \mathbb P(Y>y | W) 
  = W^\alpha  y^{-\alpha} \int_a^b  g(v, W,y)
    f_1(v) \ud v,\\
  & \mathbb P(Y>y )
  = y^{-\alpha} \int_c^d\int_a^b w^{\alpha} g(v,w,y)
    f_1(v) f_2(w)\ud v\ud w.
\end{align*}
Thus
\begin{align*}
  \frac{\PP[Y>y |X] - \PP[Y>y|W]}{\PP[Y>y]} & = \frac{W^{\alpha}\left\{g(v,W,y) - \int_a^b  g(v, W,y) f_1(v) \ud v\right\}}{\int_c^d\int_a^b w^{\alpha} g(v,w,y)
                                                f_1(v) f_2(w)\ud v\ud w}
\end{align*}
By \eqref{eq:unifCVex2} and dominated convergence, $\int_c^d\int_a^b g(v,w,y)
f_1(v) f_2(w)\ud v\ud w \to c \EE(W^\alpha)$ as $y\to\infty$. Regarding the  numerator, Cauchy's inequality implies that
\begin{align*}
  &\EE \Big| W^{\alpha}\left\{g(v,W,y) - \int_a^b  g(v, W,y) f_1(v) \ud  v\right\}\Big|\\
\leq & \sqrt{\EE W^{2\alpha}}\sqrt{\EE\left\{g(v,W,y) - \int_a^b  g(v, W,y) f_1(v) \ud  v\right\}^2}. 
\end{align*}
The right-hand side tends to zero by noting that $\EE W^{2\alpha}<\infty$ and applying the dominated convergence theorem twice to the second term. The proof is complete. 


\section{Proof of Theorem~\ref{th:SAVEx}}\label{ap:proofTIREX2principle}
  We need to show that
  \begin{equation}
    \label{eq:toshowTheorem2}
Q_e \EE[ ZZ^\top -I  \given Y>y]  \xrightarrow[y\to\supy]{} 0.    
  \end{equation}
  Notice first that from (LC)~\eqref{eq:LC} and (CCV)~\eqref{eq:CCV} it holds that
  $Q_e(\Var{Z \,|\, P_e Z} -I_p) =  - Q_e P_e = 0$. Thus also \[Q_e\EE[ZZ^\top -I_p \given P_e Z] = Q_e(\Var{Z \,|\, P_e Z} -I_p) + Q_e E[Z\,|\, P_e Z] E[Z\,|\, P_e Z]^\top = Q_e P_e = 0.\]
  As a consequence
  \begin{align}
    & Q_e \EE[(ZZ^\top - I_p)\un\{Y>y\}]  = Q_e\EE[ \EE\Big((ZZ^\top - I_p)\un\{Y>y\} \,|\, P_e Z, Y  \Big) ]\nonumber \\
    & = Q_e\EE\Big[\Big(\EE[ZZ^\top - I_p \given P_e Z, Y]  - 
      \EE[ZZ^\top -I_p \given P_e Z]\Big) \un\{Y>y\}\Big] \nonumber\\
    & = Q_e \EE\Big[\Big(\EE[ZZ^\top  \given P_e Z, Y] - \EE[ZZ^\top  \given P_e Z]\Big) \un\{Y>y\}\Big] \nonumber
\end{align}
Thus in order to show~\eqref{eq:toshowTheorem2} it is sufficient to show that  for all pair $(i,j)\in\{1,\ldots, p\}^2$, writing $p_y = \PP[Y>y]$,
\begin{equation}
  \label{eq:toshowth2-a}
  p_y^{-1} \EE\Big[\Big(\EE[Z_iZ_j \given P_e Z, Y] - \EE[Z_iZ_j  \given P_e Z]\Big) \un\{Y>y\}\Big] \xrightarrow[y\to\supy]{} 0
\end{equation}
Fixing $i,j\le p$ and   following the same path as in Theorem~\ref{th:SIR_extreme} we decompose the left-hand side of \eqref{eq:toshowth2-a} for any $A>0$ as a sum $C_1(A,y) + C_2(A,y)$ where
\begin{equation*}
  \begin{aligned}
    C_1(A,y) &=   p_y^{-1} \EE\Big[\Big(\EE[Z_iZ_j \un\{\|Z\| \le A\}\given P_e Z, Y] - \dots \\
    & \dots 
    \EE[Z_iZ_j \un\{ \|Z\| \le A\} \given P_e Z]\Big) \un\{Y>y\}\Big] \;,\\
C_2(A,y)& =     p_y^{-1} \EE\Big[\Big(\EE[Z_iZ_j \un\{\|Z\|  > A\}\given P_e Z, Y] -\dots \\
 & \dots    \EE[Z_iZ_j \un\{ \|Z\| >  A\} \given P_e Z]\Big) \un\{Y>y\}\Big]. 
  \end{aligned}
\end{equation*}
Point~(iii) 
of Proposition~\ref{prop:equi_tail_cond_exp} with 
$h = 1$ and $g(Z) = Z_iZ_j\un\{ \|Z\| \le   A\}$ ensures that $C_1(A,y)\to 0$ as $y\to\supy$ for any fixed $A$. On the other hand, using that $|Z_iZ_j|\le \frac{1}{2} (|Z_i|^2 + |Z_j|^2) \le \frac{1}{2}\|Z\|_2^2 \le c \|Z\|^2$ for some constant $c$  we may bound $|C_2(A,y)|$ as   follows, 
\begin{align*}
  | C_2(A,y) |& \le p_y^{-1}c  \EE\Big[\EE[ \|Z\|^2 \un\{\|Z\|  > A\}\given P_e Z, Y] \un\{Y>y\}\Big]  + \dots \\
              & \dots  p_y^{-1}c  \EE\Big[\EE[ \|Z\|^2 \un\{\|Z\|  > A\}\given P_e Z] \un\{Y>y\}\Big]\\
              & = p_y^{-1}c \bigg(\EE[ \|Z\|^2 \un\{\|Z\|  > A\}\un\{Y>y\}] +
                \EE\Big(\EE[ \|Z\|^2 \un\{\|Z\|  > A\}\given P_e Z] \un\{Y>y\}\Big)\bigg)\\
  &= c \EE[h_{1,A}(Z) \given Y>y] + \EE[h_{2,A}(Z) \given Y>y]. 
\end{align*}
Hence, in view of condition~\eqref{eq:unif_integrabilbity_2} for any $\epsilon>0$ there exists some $A>0$ such that $$\limsup_{y\to \supy} |C_2(A,y)| \le \epsilon,$$ whence
$\limsup_{y\to \supy} |C_2(A,y)| + |C_1(A,y)|\le \epsilon$, which shows~\eqref{eq:toshowth2-a} and completes the proof.

\section{Proofs and auxiliary results for Section~\ref{sec:estimationWithUniform}}\label{sec:technical_sec3}

  \subsection{Inverse of empirical \cdf's and order statistics}\label{ap:inverseEmpiricalCdf}
The following general fact is used on several occasions in our proofs: 
\begin{fact}\label{fact:cdf}
For $u\in (0,1], \hat H^-(u) = T_{( \lceil nu \rceil )}$ and for $z \in [0,n-1]$:
\[( \hat H(T_i) < (z+1)/n) \, \Leftrightarrow  (T_i  \leq  \hat H^-(z/n)).\]
\end{fact}
\begin{proof}
  The first statement follows from the definition of $\hat H^-$. Thus, using~\eqref{eq:H_inverse},
  \[
    \begin{aligned}
      \hat H(T_i) < (z+1)/n) & \iff T_i < \hat H^-((z+1)/n ) = T_{( \lceil z +1\rceil) } = T_{( \lceil z \rceil + 1)} \\
      &\iff T_i \le T_{( \lceil z \rceil )} =   \hat H^-(z /n).
    \end{aligned}
             \]
\end{proof}

\subsection{Vervaat's Lemma}\label{sec:vervaat}

We quote Lemma~4.3 in \cite{segers:15}, which is a variant of ``Vervaat's lemma'', i.e., the functional delta method for the mapping sending a monotone function to its inverse.

\begin{lemma}
\label{lem:Vervaat:random}
Let $G : \mathbb{R} \to [0, 1]$ be a continuous distribution function. Let $0 < r_n \to \infty$ and let $\hat{G}_n$ be a sequence of random distribution functions such that, in $\ell^\infty(\mathbb{R})$, we have $r_n ( \hat{G}_n - G) \leadsto \beta \circ G$, as $n \to \infty$, where $\beta$ is a random element of $\ell^\infty([0, 1])$ with continuous trajectories. Then $\beta(0) = \beta(1) = 0$ almost surely and as $n \to \infty$, 
\begin{equation*}
  \sup_{u \in [0, 1]} r_n |  (G \{ \hat{G}_n^-(u)\} - u ) +  ( \hat{G}_n \{ G^-(u)\} - u ) | = o_{{\mathbb P}}(1).
\end{equation*}

\end{lemma}

 \subsection{Proof of Lemma~\ref{lemma:useful_lemma_generic}}\label{sec:proof_genericLemma_1}
Because we only need to show that for any $u\in \mathbb R^p$, $v^T \widetilde \Gamma  \leadsto v ^T \widetilde  W $, to prove that $\widetilde \Gamma$ is asymptotically tight  we may  consider the case where $q=1$, i.e., $h(V)\in \mathbb R$. 
Denoting by $\psi$ the derivative of $x\mapsto \EE[h(V)^2 \un{\{U\leq x\}}]$, there exist by assumption positive constants $(c_0,\delta_0)$ such that for all $\delta\leq \delta_0$, $\psi(\delta) \leq c_0 $. Similarly, because we assume  the existence of 
$\Xi = S(0)$ (Assumption 2. in Theorem~\ref{thm:generic-pair-tailprocess}), there exist positive constants $(c_1,\delta_1)$ such that for all $\delta\leq \delta_1$, $\mathbb E [h(V)^2 \,|\, U<\delta_1 ]\leq c_1 $. 
We assume in the following argument  that $k/n \leq \delta_0 \wedge \delta_1$.

We  apply Theorem 2.11.23 in \cite{vandervaart+w:1996} (Classes of functions changing with $n$) with 
\begin{align*}
&f_{n,u}(V,U) = \sqrt{\frac{n}{k}} h(V) \un{\{U\leq uk/n\}},\\
& \mathcal F_n = \{f_{n,u} \,:\, u\in[0,1]\},\\
& F_n(V,U) = \sqrt{\frac{n}{k}} | h(V) | \un{\{U\leq k/n\}}.
\end{align*}
We start by verifying equation 2.11.21 in \cite{vandervaart+w:1996}. First, we have $$\EE[ F_{n}(V,U)^2] \leq  c_1.$$ Second, for any $\eta>0 $ and $M>0$, it holds that (for $n,k$ large enough)
\begin{align*}
  &\EE[ F_{n}(V,U)^2 \un{\{ F_n(V,U) >\eta \sqrt n \}} ] \\
  & = \left(\frac n k \right)   \EE[ |h(V)|^2  \un{\{U\leq k/n\}}    \un{\{ |h(V)| \un{\{U\leq k/n\}} >\eta \sqrt k \}} ]\\
 &\leq \left(\frac n k \right)   \EE[ |h(V)|^2 \un{\{U\leq k/n\}}  \un{\{ |h(V)|  >\eta \sqrt k \}} ]\\
 &\leq \left(\frac n k \right)   \EE[ |h(V)|^2 \un{\{U\leq k/n\}}  \un{\{ |h(V)|  >M \}} ] .
\end{align*} 
Hence 
\begin{align*}
\limsup_{n\to \infty} \EE[ F_{n}(V,U)^2 \un{\{ F_n(U) >\eta \sqrt n \}} ]\leq  S(M).
\end{align*}
But $M$ is arbitrary so the latter display is arbitrarily small. Third,  by the mean value theorem, whenever $u\leq t$, $\exists \tilde t\in (u,t)$ such that 
\begin{align*}
\EE[ (f_{n,u} (V,U) - f_{n,t}(V,U))^2 ] &= \left({\frac{n}{k}}\right) \EE[ h(V)^2 \un{\{  uk/n \leq U \leq tk/n \}} ] \\
&= \psi(\tilde t k/n) ( t-u)\\
& \leq c_0 (t-u).
\end{align*} 
This implies that
\begin{align*}
\sup_{|u-t|\leq \delta_n }  \EE[ (f_{n,u} (V,U) - f_{n,t}(V,U))^2 ] \to 0 , \text{ as } \delta_n \to 0.
\end{align*}
It remains to check the entropy condition for the class $\mathcal F_n$. Let $0<\epsilon< 1$, and denote by $u_i = i \epsilon$, $i = 0, \ldots , N$ and $u_{N+1} = 1$ with $N = \lfloor 1/\epsilon\rfloor $. Denote respectively by $f_{n,u}^+$ and $f_{n,u}^-$ the positive and negative parts of $f_{n,u}$ and by $\mathcal{F}_n^+, \mathcal{F}_n^-$ the associated classes. The functions $(f_{n,u_i }^+)$ (resp. $(f_{n,u_i }^-)$)  forms an $(\epsilon,L_2) $-bracketing of $\mathcal F_n^+ $ (resp. $\mathcal F_n^-$), i.e., for any $u\in [0,1]$, there exists $i$ such that
\begin{align*}
f_{n,u_i} ^+\leq f_{n,u}^+ \leq f_{n,u_{i+1}}^+ , 
\end{align*}
and 
\begin{align*}
\EE[ (f_{n,u_{i+1}}^+ (V, U)  - f_{n,u_i}^+(V,U) )^2 ]\leq c_0 \epsilon. 
\end{align*}
Similar inequalities remain valid for $\mathcal F_n^-$.
Hence  considering the functions $f_{n,i} = f_{n,u_i}^+  - f_{n, u_{i+1} }^- $, we have that for $u \in [u_i,u_{i+1}]$, $i = 0,\ldots, N$, 
  \[
f_{n, u}(\point) = f_{n,u}^+(\point) - f_{n,u}^-(\point) \in [f_{n,i}(\point) , f_{n, i+1}(\point) ]\,,
  \]
 thus  there exists $ C>0$ such that
\begin{align*}
\mathcal N_{[\,]} (\epsilon\|F_n\|_{L_2(P)} , \mathcal F_n , L_2(P)) \leq  C/\epsilon^2. 
\end{align*}
The entropy condition is satisfied as for all $\delta_n\to 0$,
\begin{align}
 \int_0^{\delta_n} \sqrt{ \log \mathcal N_{[\,]} (\epsilon\|F_n\|_{L_2(P)} , \mathcal F_n , L_2(P)) } \,d \epsilon  \to  0.\label{eq:entropyCondit}
\end{align}
Consequently, the process $\widetilde \Gamma$ is tight.
Finally the covariance functions at $s\le t$ are given by
\[
  \begin{aligned}
    \Cov{\widetilde \Gamma_h(s), \widetilde\Gamma_h(t)} & =
    \EE[n/k h(V)h(V)^\top \un\{U \le s k/n \}] - \dotsb \\
    &\qquad \qquad n/k \EE[h(V)\un\{U \le s k/n \}]\EE[h(V)\un\{U \le t k/n \}]\\
    & = s\EE[ h(V)h(V)^\top\given U\le s k/n ]  - \dotsb \\
   & \qquad \qquad k/n \; st \EE[h(V) \given U\le sk/n ]  \EE[h(V) \given U\le tk/n ]
  \end{aligned}
  \]
The first term in the right-hand side converges to $s\; \Xi  = (s\wedge t )\; \Xi$ while the second term goes to zero from assumption 3. in Theorem~\ref{thm:generic-pair-tailprocess}'s statement. This concludes the proof. 
 




\subsection{Proof of Lemma \ref{lemma:useful_lemma_2}}\label{sec:proofUsefulLemma2}

We apply  Lemma~\ref{lem:Vervaat:random} (Vervaat) to the distribution functions 
\begin{align*}
\hat G_n  (u) = \left\{ \begin{array}{ll}
0 &\text{ for }  u<0  \\
{\hat F_U(u  k/n )} / {\hat F_U(k/n)}& \text{ for }  0\leq u  \leq  1  \\
1 & \text{ for }  1<u
\end{array}\right.  , \qquad G(u) = \left\{ \begin{array}{ll}
0 &\text{ for }  u<0  \\
u & \text{ for }  0\leq u  \leq  1  \\
1 & \text{ for }  1<u
\end{array}\right.. 
\end{align*}
The quantile functions of $\hat G_n$ and $G$ are respectively, for any $u\in [0,1]$,
\begin{align*}
\hat G_n ^-  (u) = \frac{\hat F_U^-(u \hat F_U( k/n))}{ k/n } ,\qquad G^-(u) = u, 
\end{align*}
 Now we prove that the conditions of  Lemma \ref{lem:Vervaat:random} are satisfied with $r_n = \sqrt k $ and $\beta$ a Brownian bridge with covariance function $u_1\wedge u_2 - u_1u_2$.
 Define
\begin{align*}
  &  a_n = \frac{ k/n }{ \hat F_U( k/n)}  
\end{align*}
and write
\begin{align*}
 \sqrt k (\hat G_n  (u)  - u) &= \left(\frac{ \sqrt k}{\hat F_U(k/n)}\right) \left(\hat F_U(uk/n)   - u \hat F_U(k/n) \right)\\
&=a_n  \sqrt k  \left( \frac{n}{k} \hat F_U(uk/n)   - u\frac{n}{k}  \hat F_U(k/n) \right)\\
&=a_n  \sqrt k \left( (\frac{n}{k} \hat F_U(uk/n) - u ) -  u ( \frac{n}{k}  \hat F_U(k/n) - 1) \right)\\
&=a_n \left( \hat \gamma_1(u) -  u \hat \gamma_1(1) \right)\\
& = a_n \hat \gamma_2(u), 
\end{align*}
where $\hat\gamma_1$ is defined in~\eqref{eq:gamma1} and 
\begin{align}
  \label{eq:gamma2}
  \hat \gamma_2(u) = \hat \gamma_1(u) -  u \hat \gamma_1(1). 
\end{align}
Now use that $a_n\to 1$ in probability 
and that $\sup_{u\in[0,1] } |\hat \gamma_2(u)|= O_{\mathbb P} (1)$ (both are consequences of Corollary \ref{lemma:useful_cor_1}) to conclude (invoking Slutsky's lemma)  that
\begin{align}
\sqrt k (\hat G_n  (u)  - u)  &= \hat \gamma_2(u) + (a_n - 1 ) \hat \gamma_2(u) \nonumber  \\
& = \hat \gamma_2(u)  + o_{{\mathbb P}}(1),\label{eq:Gnu}
\end{align}
 where the stochastic convergence  $o_{{\mathbb P}}(1)$ is uniform in $ u\in [0,1]$.  In particular that $\sqrt k (\hat G_n  (u)  - u) $ weakly converges to a Brownian bridge with covariance function $u_1\wedge u_2 - u_1u_2$. The conclusion of Lemma \ref{lem:Vervaat:random} is that
\begin{align*}
\sup_{u\in (0,1] }  \left|  \hat \gamma_3 (u)  + 
  \sqrt k  (\hat G_n  (u)  - u) \right|  = o_{{\mathbb P}}(1),
\end{align*}
with 
\begin{equation}
  \label{eq:gamma3}
  \hat \gamma_3 (u)  =  \sqrt{k} ( \hat G_n^-(u) - G^-(u))  =\sqrt k \left( \frac{n}{k}\hat F_U^-(u \hat F_U( k/n))  - u \right) . 
\end{equation}
Consequently, using~\eqref{eq:Gnu}, 
\begin{align}\label{eq:gam1_2}
\sup_{u\in (0,1] }   \left|  \hat \gamma_3 (u)   + \hat \gamma_2(u) \right|  = o_{{\mathbb P}}(1).
\end{align}
Remark that 
\begin{align}\label{eq:useful}
\sqrt k \left( (n/k)\hat F_U^-(u k/n )  - u \right)   =  \hat \gamma_3 (u  a_n  ) + u \sqrt k (a_n - 1).
\end{align}
and that, as $\hat \gamma_1(1) = \sqrt k  ( (n/k) \hat F_U(k/n) - 1)$,
\begin{align}
\sqrt k (a_n - 1) & =- a_n \hat\gamma_1(1) \label{eq:an}
\end{align}
Using the triangle inequality and (\ref{eq:useful}), we get
\begin{align*}
&\left|\sqrt k   \left( (n/k)\hat F_U^-(u k/n )  - u \right) + \hat \gamma_1(u)   \right|  \\
&=  \left| \hat \gamma_3(u  a_n  )  +u \sqrt k (a_n - 1) +   \hat \gamma_1(u)\right|  \\
&\leq \left| \hat \gamma_3(u  a_n  )  + \hat \gamma_2(ua_n)  \right| +  \left| u \sqrt k (a_n - 1) +   \hat \gamma_1(u) -\hat \gamma_2(ua_n)   \right| \\
&=   \left| \hat \gamma_3(u  a_n  )  + \hat \gamma_2(ua_n)  \right|  + | \hat \gamma_1(u   ) - \hat \gamma_1(u  a_n  ) | ,
\end{align*}
where the last line is deduced from (\ref{eq:an}) and $\hat \gamma_2(u) = \hat \gamma_1 (u)- u \hat \gamma_1(1)$. Whenever $u\in [0,1/2]$, we have, with probability going to $1$, that $u  a_n \in [0,1]$. 
Moreover, because $a_n\to 0$ in probability, there exists $\delta_n\to 0$ such that the event $ |u - u a_n |\leq |a_n| \leq \delta_n$ has probability going to $1$. 
On these events, it holds
\begin{align*}
&\sup_{u\in (0,1/2]} \left| \hat \gamma_3(u  a_n  )  + \hat \gamma_2(ua_n)  \right|  \leq  \sup_{u\in (0,1]} \left| \hat \gamma_3(u   )  + \hat \gamma_2(u)  \right| =   o_{\mathbb P}(1)\\
&\sup_{u\in (0,1/2]}   | \hat \gamma_1(u   ) - \hat \gamma_1(u  a_n  ) |  = \sup_{u\in (0,1],v\in(0,1], |u-v|\leq \delta_n}   | \hat \gamma_1(u   ) - \hat \gamma_1(v  ) |  =  o_{\mathbb P}(1).
\end{align*}
We have used (\ref{eq:gam1_2})  and the asymptotic equicontinuity of $\hat \gamma_1$. Consequently we have shown that, whenever $n\to \infty $, $k\to \infty$, we have
\begin{align*}
\sup_{u\in (0,1/2]}&  \left|\sqrt k \left( (n/k)\hat F_U^-(u k/n )  - u \right) + \hat \gamma_1(u)   \right|  =o_{\mathbb P}(1) .
\end{align*}
To obtain the stated result, apply this  with $2k$ in place of $k$.



\section{Extension to non-standardized covariates}\label{sec:nonstandard}
In this section we extend our inverse regression framework to the case
of non-standardized covariates $X$.
Section~\ref{sec:background-nonstandard} recalls standard results for
that matter. In Section~\ref{sec:nonstandard-sirex-savex} the
extensions of the TIREX1 and TIREX2 principles are presented. The
proofs of these results are omitted since they follow from classical
arguments from non-standardized covariates combined with our proofs
with standardized covariates from Section~\ref{sec:tail-CI-central space}. In 
Section~\ref{sec:nonstandard-estimation} we show that estimating the
mean vector and covariance matrix for standardization does not change
the asymptotic behavior of the latter tail processes.

\subsection{SIR and SAVE principles  with non-standardized covariates}\label{sec:background-nonstandard}
We first recall some necessary background from the theory of inverse regression with non-standardized covariates, as exposed \emph{e.g.} in \cite{cook+w:1991}.

\subsubsection{SDR spaces }Recall from Section~\ref{sec:background-SIR} that in terms of non-standardized covariates $X = m + \Sigma^{1/2}Z$, a subspace $\tilde E$ of $\rset^p$   is a SDR space for the pair $(X,Y)$ if and only if $\tilde E = \Sigma^{-1/2} E$ where  $E$ is a SDR space for the pair $(Z,Y)$. 
We denote in the sequel by $\tilde P$ the  orthogonal projector onto such a  SDR space $\tilde E$ and we define  $\tilde Q = I_p - \tilde P$. 


\subsubsection{Linearity and constant variance conditions }
Conditions LC~\eqref{eq:LC} and CCV~\eqref{eq:CCV} regarding the standardized variable $Z$ are respectively equivalent to
  \begin{equation}
    \label{eq:LC-X}
    \EE [X| \tilde P X]  = b + B \tilde P X 
  \end{equation}
for some $b\in \mathbb R^p$ and $B\in \mathbb R^{p\times p}$,  and 
  \begin{equation}
    \label{eq:CCV-X}
\Var{X|\tilde P X }    \text{ is constant} \quad \text{a.s.}
  \end{equation}

\subsubsection{ SIR principle and CUME matrix }
  The extension of the SIR principle (Proposition~\ref{prop:sir_principle}) in terms of non-standardized covariates, is that under condition~\eqref{eq:LC-X}, it holds that
  \begin{equation}
    \label{eq:SIRprinciple-nonstd}
  \Sigma^{-1} (\EE[X | Y] - m)   \in \tilde E.
  \end{equation}
As a consequence  the  CUME matrix defined in~\eqref{eq:M-cume} must be replaced with 
the matrix  $\tilde M _{\text{CUME}}  = \EE [\tilde m(Y) \tilde m(Y) ^T] $, with
\[\tilde m(y) = \EE[ (X -m) \un\{ Y\le y\}], \]
in which case it holds that
\[\Span(\tilde M_{\text{CUME}} ) \subset \Sigma \tilde E = \Sigma^{1/2} E. \]

\subsubsection{SAVE principle }
The parallel statement of Proposition~\ref{prop:SAVE} is that under conditions~\eqref{eq:LC-X} and \eqref{eq:CCV-X}, we have
\begin{equation}
  \label{eq:SAVE-princple-nonstd}
  \Span(\Sigma ^{-1} ( \Var{X  \given Y} -\Sigma ) )\subset \tilde E  \quad a.s., 
\end{equation}
   or equivalently $\Span(\Sigma ^{-1} ( \EE\big[ (X- m ) (X - m )^\top  \,|\, Y \big] -\Sigma ) )\subset \tilde E $. 


\subsection{TIREX principles with non-standardized covariates}\label{sec:nonstandard-sirex-savex}
It follows from Definition~\ref{as:tailCI} that $E_e$ is an extreme SDR space for the pair $(Z,Y)$ if and only if $\tilde E_e = \Sigma^{-1/2} E_e$ is an extreme SDR space for the pair $(X,Y)$, in the sense that, denoting by $\tilde P_e$ the orthogonal projection on $\tilde E_e$, $Y_\infty\indep \, X \,|\, \tilde P_e X$.  
 
We now state the analogue statement to Theorem~\ref{th:SIR_extreme} in  terms of the non-standardized covariate $X$.
\begin{proposition}[non-standardized TIREX1 principle]\label{prop:nonstandard-SIREX}
The assumptions of Theorem~\ref{th:SIR_extreme} are equivalent to
  \begin{enumerate}
  \item $\lim_{A\to \infty }\limsup_{y\to y^+} \EE[ \tilde g_{k,A}(X) \given Y>y ]  = 0 $, $k = 1,2 $
    where $\tilde g_{1,A}(X) =  \|X\| \un\{\|X\| >A\}  $ and $\tilde g_{2,A}(X) = \EE[ \|X\| \un\{\|X\| >A\} \given \tilde P_eX]$, where $\tilde P_e$ is the orthogonal projector on $\tilde E_e = \Sigma^{-1/2} E_e$.

  \item The covariate vector satisfies the non-standardized linearity condition~\eqref{eq:LC-X}
    \item For some $\tilde \ell\in\rset^p $,  with $m = \EE[X]$
  \begin{equation}
    \label{eq:limitEZ-largeY}
      \EE[X \given  Y>y ]  - m \xrightarrow[y\to y^+]{}\tilde \ell. 
    \end{equation}
  \end{enumerate}
  In such a case $\tilde \ell = \Sigma^{1/2} \ell$ where $\ell$ is the limit defined in Theorem~\ref{th:SIR_extreme}  and the conclusion   is that
  \[\Sigma^{-1} \tilde \ell \in \tilde E_e. \] 
\end{proposition}

\begin{proposition}[non-standardized TIREX2 principle]\label{prop:nonstd-SAVEX}
  Assume that $(X,Y)$ and the extreme SDR space satisfy the assumptions of Proposition~\ref{prop:nonstandard-SIREX} (non-standardized TIREX1 principle) and that in addition,
  \begin{enumerate}
  \item \ (second order uniform integrability):
\begin{align}
&\lim_{A\to \infty }\limsup_{y\to y^+} \EE[ \tilde h_{k,A}(X) \given Y>y ]  = 0 ,\qquad k = 1,2 \;,\label{eq:unif_integrabilbity_2_nonstandard}
\end{align}
where  
 $\tilde h_{1,A} (X) =  \|X\|^2 \un\{\|X\| >A\}  $ and $\tilde h_{2,A}(X) = \EE[ \|X\|^2 \un\{\|X\| >A\} \given \tilde P_eX]$, 

\item \ (CCV) The covariate vector $X$ satisfies the  non-standardized constant variance condition~\eqref{eq:CCV-X}  relative to~$\tilde P_e$, 
\item \ (Convergence of conditional expectations)  For some $\tilde S\in\rset^{p\times p} $,
  \begin{equation}
    \label{eq:limitEZ2-largeY}
      \EE[XX^\top \given  Y>y ]\xrightarrow[y\to \supy]{} \tilde S + \tilde \ell \tilde \ell^\top, 
    \end{equation}
where $\tilde \ell$ is the limit appearing in Proposition~\ref{prop:nonstandard-SIREX}. 
  \end{enumerate}
  Then  $$\Span( \Sigma^{-1} (\tilde S - \Sigma)) \subset \tilde E_e, $$
  \ie  $\tilde Q_e \Sigma^{-1} (\tilde S  - \Sigma) =0.$ 
\end{proposition}

\subsection{Estimation with non-standardized covariates}\label{sec:nonstandard-estimation}
Consider the non-standardized versions of the matrices $M_{\text{TIREX1}},M_{\text{TIREX2}}$ from Section~\ref{sec:estimationWithUniform} defined as follows:
  \begin{equation}
    \label{eq:Cndeterministic-nstd}
    \begin{gathered}
    \tilde{ M}_{\text{TIREX1}} = \int_0^1C_n^{m}(u)C_n^{m}(u)^\top  \ud u \,,\text{  with  }\\
    C_n^m(u) = \frac{n}{k}\EE[(X-m)\un\{ \tilde Y < F^-(uk/n)\}] \;,           
    \end{gathered}
  \end{equation}
and 
  \begin{equation}
    \label{eq:Bndeterministic-nstd}
    \begin{gathered}
      \tilde{ M}_{\text{TIREX2}} = \int_0^1 B_n^{m,\Sigma}(u)B_n^{m,\Sigma}(u)^\top  \ud u \,,\text{ with }  \\
      B_n^{m,\Sigma}(u) = \frac{n}{k}\EE[\big( (X-m)(X-m)^\top - \Sigma \big)\un\{ \tilde Y < F^-( uk/n)\} ] .        
    \end{gathered}
  \end{equation}
  In view of Propositions~\ref{prop:nonstandard-SIREX} and~\ref{prop:nonstd-SAVEX}, under the same assumptions therein, $\Span(\tilde M_{\text{TIREX1}})$ and $\Span(\tilde M_{\text{TIREX2}})$ become close to $\Sigma \tilde E_e$ as $n\to\infty$, in the sense that
  \[\lim_{n\to \infty }  \tilde Q_e\Sigma^{-1} \tilde M_{\text{TIREX1}}
    = \lim_{n\to \infty }\tilde Q_e \Sigma^{-1} \tilde M_{\text{TIREX2}} = 0, \]
  where $\tilde Q_e$ is the orthogonal projector on $\tilde E_e^\perp$. 

  Notice that we can write $C_{n}^{m},B_n^{m,\Sigma}$  in terms of $C_n, B_n$ as follows:
  \begin{equation}
    \label{eq:linkCnBn-nonstandard}
    \begin{aligned}
      C_n^{m}(u) & = \Sigma^{1/2}C_n(u) \\
      B_n^{m,\Sigma}(u) &= \Sigma^{1/2}B_n(u) \Sigma^{1/2}
    \end{aligned}
  \end{equation}
  Despite the apparent simplicity of~\eqref{eq:linkCnBn-nonstandard}, in the estimation step  with unknown covariate's mean and covariance, one must replace $m$ and $\Sigma$ in definitions~\eqref{eq:Cndeterministic-nstd} and~\eqref{eq:Bndeterministic-nstd} with some estimates, \emph{e.g.} the  empirical ones which we denote by $\hat m, \hat\Sigma$. Namely we consider the processes
  \begin{equation}
    \label{eq:empiricalCnBn-nonstd}
    \begin{aligned}
      \hat C_n^{\hat m}(u)
      & =\frac{1}{k}\sum_{i=1}^n(X_i- \hat m)\un\{ \tilde Y_i \le \hat F^-(uk/n)\} \;,\\ 
      \hat B_n^{\hat m,\hat \Sigma}(u)
      &= \frac{1}{k}\sum_{i=1}^n \Big( (X_i-\hat m)(X_i-\hat m)^\top - \hat \Sigma\Big)
      \un\{ \tilde Y_i \le  \hat F^-( uk/n)\} ]
    \end{aligned}
  \end{equation}
and define the non-standardized TIREX1 and TIREX2 tail empirical processes respectively as 
\begin{align}
  \sqrt{k} \Big(\hat C_n^{\hat m} - C_n^{m} \Big)
\text{ and  }
\sqrt{k} \Big(\hat B_n^{\hat m,\hat \Sigma} - B_n^{m, \Sigma} \Big) . 
    \label{nonstand-sirexSavexProcesses}
\end{align}

  We assume that the conditions for the central limit theorem regarding the estimators $\hat m$ and $\hat \Sigma$ are met. For instance, we assume that $X$ admits fourth order moments, an assumption which is needed anyway for the weak convergence of the TIREX2 process, see Corollary~\ref{cor:weakCV-SIREX-SAVEX}. Thus we work under the assumption that
  \begin{equation}
    \label{eq:cvMeanVarianceRootN}
    \hat m = m + O_{\PP}(1/\sqrt{n}) \;;\quad \hat \Sigma = \Sigma + O_{\PP}(1/\sqrt{n} ). 
  \end{equation}

  \begin{proposition}[Weak convergence of non-standardized TIREX processes]\label{prop:weakCV-nonstd-sirexSavex}
    Under Assumption~\eqref{eq:cvMeanVarianceRootN},
    \begin{enumerate}
    \item The standardized TIREX1 process $\sqrt{k}(\hat C_n - C_n)$ converges weakly  in $\ell^\infty([0,1])$  to a tight Gaussian process $W_{1}$ if and only if its non-standardized version defined in~\eqref{nonstand-sirexSavexProcesses} converges weakly, in the same space,  to the Gaussian process $\Sigma^{1/2}W_1$.
      \item If weak convergence of the TIREX1 process holds true, then  the standardized TIREX2 process 
       $\sqrt{k}(\hat B_n - B_n)$  converges weakly in $\ell^\infty([0,1])$ to a tight Gaussian process $W_{2}$ if and only if its non-standardized version defined in~\eqref{nonstand-sirexSavexProcesses} converges weakly, in the same space,  to the Gaussian process $\Sigma^{1/2}W_2\Sigma^{1/2}$.
    \end{enumerate}
  \end{proposition}
  \begin{proof}[Proof of Proposition~\ref{prop:weakCV-nonstd-sirexSavex}]~\\~
    \begin{enumerate}
    \item 
    Substituting $X-m$ with $\Sigma^{1/2}Z$ we obtain
    \begin{align}
      \hat C^{\hat m}(u) &= \frac{1}{k} \sum_{i=1}^n \Sigma^{1/2}(Z_i + m - \hat m)
                           \un\{\tilde Y_i \le  \hat F^-( uk/n) \}  \nonumber \\
                         & = \Sigma^{1/2} \left\{ \hat C_n(u) +  \Delta_n(u)  (m - \hat m) \right\}
                           \label{eq:decompose-HatCnm}
    \end{align}
    where $\hat C_n$ is defined in~\eqref{eq:hatCn} in terms of $Z$ and
    \begin{align}
      \Delta_n (u) &:= \frac{n}{k} \hat F\big( \hat F^-( u k/n)\big) 
      \le \frac{n}{k} \hat F\big( \hat F^-(  k/n)\big) 
                        = \frac{n}{k} \hat F(\tilde Y_{(k)}) 
       = 1       \label{eq:Delta-le-1}                   . 
      \end{align}
     Combining the  latter upper  bound, \eqref{eq:decompose-HatCnm} and~\eqref{eq:linkCnBn-nonstandard} we obtain 
      \begin{align}
        \sqrt{k}\left( \hat C_n^{\hat m} - C_n^m\right)
        &=   \Sigma^{1/2} \sqrt{k}( \hat C_n(u) - C_n(u)) + R_n(u), 
          \label{eq:equiv-nonstand-stand-Cn}
      \end{align}
      where $\sup_{u\in[0,1]} R_n(u) = O_{\PP}(\sqrt{k/n}) = o_{\PP}(1) $ and the main term $\sqrt{k}( \hat C_n(u) - C_n(u))$ is the standardized TIREX1 process. The first assertion of the statement follows from the Slutsky's lemma.

\item 
      The argument for the second order method is similar though the computation is more involved. We have   
      \begin{align*}
        \hat B_n^{\hat m, \hat\Sigma}(u)
        & =\Sigma^{1/2}\bigg\{ \frac{1}{k}\sum_{i\le n} \Big(\big(Z_i + \Sigma^{-1/2}(m - \hat m)\big)
          \big(Z_i + \Sigma^{-1/2}(m - \hat m)\big)^\top  - \Sigma^{-1/2} \hat \Sigma\Sigma^{-1/2}\Big) \times\dotsb\\
   &\qquad \qquad \dotsb       \un\{\tilde Y_i \le  \hat F^-( uk/n) \}\bigg\}\Sigma^{1/2}  \\
        &= \Sigma^{1/2}\left\{ \hat B_n(u) + A_{1,n}\Delta_n(u)  +
       A_{2,n}(u) \right\} \Sigma^{1/2}
      \end{align*}
      with $\Delta_n(u)\le 1$ as in~\eqref{eq:Delta-le-1} and
      \begin{align*}
        A_{1,n} & = \big(I_p - \Sigma^{-1/2}\hat\Sigma\Sigma^{-1/2}\big) +
                  \Sigma^{-1/2} (m - \hat m)(m - \hat m)^\top \Sigma^{-1/2}\,,  \\
        A_{2,n} &=   \Sigma^{-1/2}(m - \hat m)\hat C_n^\top(u) + \hat C_n(u)(m - \hat m)^\top \Sigma^{-1/2}.
      \end{align*}
      Under the assumption that the TIREX1 empirical process converges weakly we have that $\sup_u \hat C_n(u) = O_{\PP}(1)$, and using~\eqref{eq:cvMeanVarianceRootN} and~\eqref{eq:linkCnBn-nonstandard} we obtain
      \begin{align*}
        \sqrt{k}\Big(  \hat B_n^{\hat m, \hat\Sigma}(u) - B_n^{m,\Sigma}(u)\Big)
        &= \Sigma^{1/2} \sqrt{k}\Big(  \hat B_n(u) - B_n(u)\Big) \Sigma^{1/2} + R'_n(u)
      \end{align*}
      with $\sup_u R'_n(u) = O_{\PP}(\sqrt{k/n}) = o_{\PP}(1)$. The  second assertion follows. 
          \end{enumerate}
       \end{proof}

\end{document}